\documentclass[10pt, english]{amsart}
\usepackage{amssymb,amsthm,amsmath,amstext,graphicx}
\usepackage{mathrsfs}
\usepackage{bm}
\usepackage[utf8]{inputenc}
\usepackage{amscd}
\usepackage[all]{xy}
\usepackage{comment}
\usepackage{nicematrix}
\usepackage{todonotes}
\usepackage{booktabs}
\usepackage{caption} %to put some space between tables and caption
\usepackage{enumitem}
\usepackage[breaklinks]{hyperref} %hyperlinks
\hypersetup{
	colorlinks,
	linkcolor={red},
	citecolor={blue},
	urlcolor={red}
}
\usepackage[nameinlink,capitalise]{cleveref}

\newcommand{\laurent}[1]{\todo[inline, size=\tiny, author=Laurent, backgroundcolor=green!20]{#1}}

\newcommand{\marcello}[1]{\todo[inline, size=\tiny, author=Marcello, backgroundcolor=blue!20]{#1}}
\newcommand{\giovanni}[1]{\todo[inline, size=\tiny, author=Giovanni, backgroundcolor=purple!40]{#1}}

\theoremstyle{plain}
\newtheorem{theorem}[table]{Theorem}
\newtheorem{remark}[table]{Remark}
\newtheorem{corollary}[table]{Corollary}
\newtheorem*{conjecture*}{Conjecture}
\newtheorem*{theorem1}{Main Theorem A}
\newtheorem*{theorem2}{Main Theorem B}
\newtheorem{prop}[table]{Proposition}
\newtheorem{lemma}[table]{Lemma}
\newtheorem{prob}[table]{Problem}
\crefname{equation}{}{}
\theoremstyle{definition}
\newtheorem{definition}[table]{Definition}

\def\CC{{\mathbb{C}}}
\def\FF{{\mathbb{F}}}

\def\PP{{\mathbb{P}}}
\def\QQ{{\mathbb{Q}}}\def\ZZ{{\mathbb{Z}}}

\def\cO{{\mathcal{O}}}
\def\oo{{\mathcal{O}}}

\def\cE{{\mathcal{E}}}
\def\cF{{\mathcal{F}}}\def\cP{{\mathcal{P}}}
\def\cG{{\mathcal{G}}}
\def\cL{{\mathcal{L}}}

\def\U{{\mathcal{U}}}
\def\cN{{\mathcal{N}}}
\def\cU{{\mathcal{U}}}
\def\cQ{{\mathcal{Q}}}

\def\ra{{\rightarrow}}
\def\lra{{\longrightarrow}}

\def\Fl{{\mathrm{Fl}}}
\def\Gr{{\mathrm{Gr}}}
\def\HK{hyperK\"ahler}

\title{Even nodal surfaces of K3 type}

\author{Marcello Bernardara}
\address{Institut de Math\'ematiques de Toulouse \\ UMR 5219 \\ Universit\'e 
de Toulouse \\ 
Universit\'e Paul Sabatier \\ %
118 route de Narbonne \\ %
31062 Toulouse Cedex 9\\ %
France}
\email{marcello.bernardara@math.univ-toulouse.fr}

\author{Enrico Fatighenti}
\address{
Alma Mater Studiorum - Universit\`a di Bologna\\ Dipartimento di Matematica \\ Piazza di porta san Donato 5\\ 40126 Bologna.}
\email{enrico.fatighenti@unibo.it}

\author{Grzegorz Kapustka}
\address{Department of Mathematics and Informatics \\ Jagiellonian University \\
 {\L}ojasiewicza 6 \\ 30-348, Krak\'ow \\ Poland}
\email{grzegorz.kapustka@uj.edu.pl}

\author{Micha{\l}  Kapustka}
\address{Institute of Mathematics of the Polish Academy of Sciences \\ ul.\ \'Sniadeckich
8 \\ 00-656 Warszawa \\ Poland
}
\email{michal.kapustka@impan.pl}

\author{Laurent Manivel}
\address{Institut de Math\'ematiques de Toulouse \\ UMR 5219 \\ Universit\'e 
de Toulouse \\ CNRS, Universit\'e Paul Sabatier \\ %
118 route de Narbonne \\ %
31062 Toulouse Cedex 9\\ %
France}
\email{laurent.manivel@math.cnrs.fr}

\author{Giovanni Mongardi}
\address{
Alma Mater Studiorum - Universit\`a di Bologna\\ Dipartimento di Matematica \\ Piazza di porta san Donato 5\\ 40126 Bologna.}
\email{giovanni.mongardi2@unibo.it}

\author{Fabio Tanturri}
\address{Dipartimento di Matematica \\
Dipartimento di Eccellenza 2023-2027\\
Universit\`a di Genova\\
Via Dodecaneso 35\\
16146 Genova, Italy}
\email{tanturri@dima.unige.it}

\date{}

\begin{document}

\begin{abstract}
   We study Fano fourfolds of K3 type with a conic bundle structure. We construct direct geometrical links between these fourfolds and hyperK\"ahler varieties. As a result we describe families of nodal surfaces that can be seen as generalisations of Kummer quartic surfaces. Each of these families actually arises through two 
   families of Fano fourfolds, whose conic bundle structures are 
   related by hyperbolic reduction. 
\end{abstract}

\maketitle

\section{Introduction}

%\fabio{Abstract to be written down}
%\fabio{Grzegorz and Micha{\l} should check in the last page if their affiliations are correct}
Quartic surfaces with the maximal number of nodes % i.e., $16$ nodes,
are Kummer quartics. Their study is a classical and fascinating topic, 
related to K3 surfaces.
Kummer quartics were introduced by  Ernst Kummer in \cite{K} as  quotients of the Jacobian of a genus $2$ curve by the involution $x\mapsto -x$. This involution has $16$ fixed points, the two-torsion points of the Jacobian, which give rise to the $16$ nodes of the quartic surface. 

The maximal number of nodes on a quintic surface is $31$ (see \cite{beauville}) and is realised by the Togliatti surface but cannot be obtained as a quotient of a smooth surface. 
Other examples of nodal surfaces were described in \cite{cc} as  degeneracy loci of symmetric morphisms (we call them surfaces with an even set of nodes); as a consequence they can be obtained as quotients of smooth surfaces. There is a classification of hypersurfaces with an even set of nodes of small degree but a general classification seems out of reach. 
%\laurent{Say a bit more about the classical literature?}

On the other hand, in the recent \cite{hu21}, Huybrechts links some nodal quintic surfaces related to K3 surfaces and hyperK\"ahler manifolds in a very fascinating way, that is, via cubic fourfolds. Indeed, if $X_3 \subset \PP^5$ is a cubic hypersurface, and $\ell \subset X_3$ a general line, then resolving the projection $X_3 \dashrightarrow \PP^3$ from $\ell$, one obtains a conic bundle $X \to \PP^3$ whose discriminant is a quintic surface with $16$ nodes. As shown by Catanese \cite{ca} (see \cite[Prop.\ 115]{cataneseal} for the statement in this form), all quintic surfaces in $\PP^3$ with an even set of $16$ nodes arise from such a pair $(X_3,\ell)$ as discriminant $\Delta$ of this conic bundle. Moreover, there exists a double cover $F \to \Delta$, where $F$ is a smooth surface of general type, with a natural $\PP^1$-fibration $Z \to F$ parametrising vertical lines in the conic bundle structure. As observed in \cite{hu21}, this allows us to compare %via $H^4(X,\ZZ)$,
the middle cohomology $H^4(X_3,\ZZ)$ with the subspace $H^2(F)^- \subset H^2(F,\ZZ)$ which is anti-invariant with respect to the involution given by the double cover. The upshot is that $H^2(F)^-$ has a K3 structure
\cite[Theorem 0.1]{hu21}.

In this paper we are interested in studying other surfaces with a similar property that we will call even nodal surfaces of K3 type.

\begin{definition}
Let $S$ be a surface with an even set of nodes that appears as a quotient of $F$ by an involution.  We say that $S$ is an \emph{even nodal surface of $K3$ type} if the Hodge structure on $H^2(F)^-$ is of K3 type. 
\end{definition}

The first result of the paper, \cref{prop:cbisometry}, provides a way to construct even nodal surfaces of $K3$-type. More precisely,  whenever a smooth fourfold admits a conic bundle structure over a smooth threefold, then the discriminant locus of that conic bundle structure is an even nodal surface by a general result of Barth. \Cref{prop:cbisometry} states then that whenever the fourfold is a Fano fourfold of K3 type with  conic fibration then the associated discriminant locus is an even nodal surface of K3 type.

The remaining part of the paper is devoted to the study of examples of even nodal surfaces of K3 type arising from \cref{prop:cbisometry}  and their geometric relation with associated Fano and hyperK\"ahler manifolds. 

In particular, in this context, the rich geometry and the tight relations between Fano fourfolds of K3 type and hyperK\"ahler varieties lead to wonder whether the surface $F$ with its involution is nothing but a section of a natural hyperK\"ahler variety associated to the related Fano fourfold of K3 type, endowed with a non-symplectic involution.

Summing up, for fourfolds of K3 type with a conic bundle structure, 
it is natural to think of the following question.

\begin{prob}\label{prob:LLSV}
Let $F \to \Delta$ be the discriminant double cover associated to a conic bundle
$X \to B$ such that $X$ is a Fano fourfold of K3 type.
Is $F$ a section of a hyperK\"ahler manifold $M$ associated to $X$,  the involution on $F$ being induced by a non-symplectic involution on $M$?
\end{prob}

A situation similar to \cite{hu21} occurs in at least two more examples. Let us consider Gushel--Mukai fourfolds, i.e.\ transverse intersections of a cone over the Grassmannian $\Gr(2,5)$ with a hyperplane and a quadratic hypersurface. Such fourfolds $X_{10} \subset \Gr(2,5)$ admit actually two different conic bundle structures. The first one has base $\PP^3$ and is obtained by blowing up $X_{10}$ along a del Pezzo surface of degree four. The discriminant $\Delta \subset \PP^3$ is a sextic surface with $40$ nodes. The second one has for base a smooth quadric threefold $\QQ^3$, and is obtained by blowing up $X_{10}$ along a $\sigma$-quadric. The discriminant $\Delta \subset \QQ^3$ is a surface of degree $4$ in $\QQ^3$ (degree $8$ in $\PP^4$) 
with $20$ nodes. In both  cases, the discriminant double cover $F \to \Delta$ is a smooth surface of general type with an involution, and one can construct a K3 type Hodge structure of $H^2(F)^-$ via the geometric correspondence with $H^4(X_{10},\ZZ)$. The three examples are summarised in \cref{summaryTable}.

\begin{table}[h!bt]
\caption{Three Fano fourfolds $X$ of K3 type with a conic bundle structure $X \to B$ with discriminant $\Delta$.}
\label{summaryTable}
\begin{tabular}{ccc} \toprule
%conic bundle & discriminant \\
%\midrule
$X$ & \rule{10pt}{0pt}$B$\rule{10pt}{0pt} & $\Delta$ \\
					\midrule
				 Blow-up of $X_3$ along a line & $\PP^3$ & quintic with 16 nodes\\
				Blow-up of $X_{10}$ along a $dP_4$ & $\PP^3$ & sextic with 40 nodes \\
    Blow-up of $X_{10}$ along a $\sigma$-quadric & $\QQ^3$ & quartic with 20 nodes \\
					\bottomrule
				\end{tabular}
\end{table}

One of the results of this paper is a positive answer to \cref{prob:LLSV} in two of the three cases listed in \cref{summaryTable}.

\begin{theorem1}
\hypertarget{main_theorem_a}{}
Let $X_{10} \dashrightarrow B$ be one of the conic bundles above, and $F \to \Delta$ its discriminant double cover. Then $F$ is a section of a hyperK\"ahler manifold $M$ and the involution on $F$ is induced by a non-symplectic involution on $M$. Specifically:
\begin{itemize}[leftmargin=*]
    \item if $\Delta \subset \PP^3$ is a sextic with $40$ nodes, then $M$ is a double EPW sextic and $F$ is a double linear section of $M$;
    \item if $\Delta \subset \QQ^3$ is a quartic with $20$ nodes, then $M \subset \Gr(3,6)$ is a double EPW cube and $F$ is a section of $M$ by a special $\Gr(2,4)$.
\end{itemize}
In both cases, the double cover $F \to \Delta$ is the restriction of the double cover structure of $M$. Furthermore, the surfaces $F$ are smooth,  algebraically simply connected minimal surfaces of general type, whose numerical invariants are explicitly computed.
\end{theorem1}

Recently, a list of families of Fano fourfolds with K3 type Hodge structure in their middle cohomology was provided in \cite{bfmt}. Surprisingly, in one of the families, called in loc.cit.\ R-62, each member $X'$ arises as a resolution of a singular $(2,2)$ divisor in $\PP^2\times \PP^3$, and admits a conic bundle structure $X' \to \PP^3$ with discriminant $\Delta$ a quintic surface with an even set of $16$ nodes, hence defining a double cover $F \to \Delta$. Catanese's result implies then that $F$ and its involution are also associated to a cubic fourfold. The interesting fact is that the family R-62 from \cite{bfmt} comes with one modulus less than the family of cubic fourfolds, leading to suppose that a $F \to \Delta$ arising from a conic bundle $X' \to \PP^3$, and the K3 structure in $H^2(F)^-$ henceforth, are of some special type.

Even more interestingly, two families of Fano fourfolds from \cite{bfmt}, called in loc.cit.\ K3-31 and K3-35, arising both as blow-ups of smooth rational fourfolds (the intersection of two quadrics and a quadric respectively) along a genus 6 K3 surface, admit conic bundle structures on $\PP^3$ and $\QQ^3$ respectively, with the same type of discriminant as the two conic bundles constructed from $X_{10}$. Both families have one modulus less than $X_{10}$. Even though for these two discriminants, no equivalent of the results of  Catanese for quintics are known, it is natural to suppose that these two families correspond to special GM fourfolds and K3 structures henceforth. Moreover, one can raise a similar question to \cref{prob:LLSV}, asking whether $F$ with its involution is a section of a hyper-K\"ahler manifold with a non-symplectic involution.

The second main result of this paper is about proving that indeed in the three pairs of families of conic bundles, at least at the level of associated even nodal surfaces of K3 type, one family can be seen as a special subfamily of the other one, and that the Hodge K3 structures on $H^2(F)^-$ are related to the one of the cubic and the GM fourfold respectively. 

Let us fix some notation. We denote by $X \to B$ and $X' \to B$ two families of conic bundles on a threefold $B$, with discriminant divisor $\Delta$ and discriminant double cover $F \to \Delta$. We summarise the three cases in \cref{summaryTable2}.

\begin{table}[h!bt]
\caption{Each of the Fano fourfold $X' \to B$ is associated to a special Fano fourfold $X \to B$ from  \cref{summaryTable} with same discriminant  (labelling from \cite{bfmt}).}
\label{summaryTable2}
\begin{tabular}{ll} \toprule
$X$ general and $X'$ special & label \\ \midrule
%conic bundle & discriminant \\
%\midrule
$X \to X_3$ blow-up along a line & C-4 \\
$X' \to Z$ resolution, $Z$ being a singular $(2,2)$-divisor in $\PP^2 \times \PP^3$  & R-62 \\
\midrule
$X \to X_{10}$ blow-up along a $dP_4$ & GM-20\\
$X' \to \QQ^5 \cap \QQ^5 \subseteq \PP^6$ blow-up along a genus 6 K3 & K3-31\\
\midrule
$X \to X_{10}$ blow-up along a $\sigma$-quadric & GM-21\\
$X' \to \Gr(2,4)$ birational with base locus a genus 6 K3 with 2 nodes & K3-35\\
%$(\deg \Delta,\mbox{nodes})$ & \rule{3pt}{0pt}$B$\rule{3pt}{0pt} & $\begin{array}{c}\mbox{general } X \\ \mbox{special } X'\end{array}$ & label \\
%					\midrule
%				 $(5,16)$ & $\PP^3$ & $\begin{array}{c}\mbox{
 %    Blow-up of $X_3$ along a line
  %                   } \\ \mbox{
   %  resolution of a singular $(2,2)$-divisor in $\PP^2 \times \PP^3$     
   %              }\end{array}$ & $\begin{array}{c}\mbox{{\bf(C-4)}}\\ \mbox{{\bf(R-62)}}\end{array}$\\
%\midrule
%			$(6,40)$\\
 %   $(4,20)$ \\
					\bottomrule
				\end{tabular}
\end{table}

\begin{theorem2}
\hypertarget{main_theorem_discriminants}{}
    For each $X'\to B$ as in \cref{summaryTable2} with associated discriminant double cover $F\to \Delta$  there exists a special fourfold $X\to B$ of type  given by \cref{summaryTable2} with isomorphic discriminant double cover $F\to \Delta$. In that case, the anti-invariant cohomology $H^2(F)^-$ contains the K3 Hodge substructure of $H^4(X,\ZZ)$, and also of $H^4(X',\ZZ)$.
Explicitly the subfamily of special fourfolds $X \to B$ correspond to:
\begin{description}
    \item[(C-4 / R-62)] cubic fourfolds containing a plane,
    \item[(GM-20 / K3-31)]  Gushel--Mukai fourfolds which are period partners of Gushel Mukai fourfolds containing a $\sigma$-plane,
    \item[(GM-21 / K3-35)] Gushel--Mukai fourfolds containing a $\tau$-quadric. 
\end{description}

\end{theorem2}

An interesting fact is that the relation described by \hyperlink{main_theorem_discriminants}{Main Theorem B}, in some cases, extends to a direct geometric relation between the fourfolds $X$ and $X'$. The latter then goes through a construction of quadric threefold fibrations by extending with a hyperbolic form the ``general'' conic bundles $X \to B$. In special cases, that can be identified geometrically, these quadric threefold fibrations admit a special isotropic bundle, through which $X' \to B$ is also realised by hyperbolic reduction. The recently developed theory of hyperbolic reductions of quadric bundles (see \cite{auel-berna-bolo,BKK,kuznethyperbolic}), along with a theorem on Hodge structures of conic bundles over threefolds, allow us to give a full description.

\begin{theorem}
\hypertarget{main_theorem_b}{}
For each of the pairs $(X,X')$ of  \hyperlink{main_theorem_discriminants}{Main Theorem B} in the cases C4/R-62 or GM-21/K3-35, there exists a quadric threefold fibrations $Y \to B$ such that:
\begin{enumerate}[leftmargin=*]
    \item Both conic fibrations $X' \to B$ and $X \to B$ are obtained by hyperbolic reduction from $Y \to B$.
% \item There exists a codimension one (codimension 2 in the GM-20 / K3-31 case) subfamily of $Y' \to B$ in the family of $Y \to B$ such that each $X' \to B$ (in a codimension one subfamily in the GM-20 / K3-31 case) is obtained by hyperbolic reduction from a $Y' \to B$.
 \item The fourfolds $X$ and $X'$ are $K(B)$-birationally equivalent.
 
\end{enumerate}
\end{theorem}
The situation in the case GM-20/K3-31 is slightly more complicated. In particular, for a general $X'\to \PP^3$ of type K3-31 there are in general six fourfolds of GM-20 type as in \hyperlink{main_theorem_discriminants}{Main Theorem B} i.e.\ having the same discriminant surface as $X'\to \PP^3$ but in general none of them will be related to $X'$ by a sequence of hyperbolic reductions/extensions. However, the relation involving hyperbolic reduction still holds between a codimension one subset of the family K3-31 and a codimension two subset of the family GM-20. Indeed, let GM-20' denote the codimension two family of fourfolds of type GM-20 corresponding to Gushel--Mukai fourfolds containing a $\sigma$-plane. Let
K3-31' be the codimension one subfamily of fourfolds of type K3-31 corresponding to singular complete intersections of two quadrics resolved by the blow up. 

\begin{theorem}
\hypertarget{main_theorem_b_K3-31}{}
General elements of the two families GM-20' and K3-31'
can be associated into pairs $(X,X')$, for which there exists a sixfold  $Y$ with a quadric threefold fibration $Y\to \PP^3$ such that both conic fibrations $X\to \PP^3$ and $X'\to \PP^3$ appear as hyperbolic reductions of $Y\to \PP^3$. 
In particular, the fourfolds $X$ and $X'$ are $K(\PP^3)$-birationally equivalent. 
\end{theorem}

Even though we follow a similar path for the three pairs of examples, both Theorem A and Theorem B will be proved on a case by case analysis due to the specificity of each example.

\medskip

The structure of the paper is as follows. In \cref{generalitiesOnConic} we recall hyperbolic reduction for quadric bundles, and we construct an isometry between the primitive middle cohomology of a conic bundle over a threefold and the anti-invariant part of the middle cohomology of the discriminant double cover. In \cref{sect16nodal,sect40nodal,sect20nodal}, we prove \hyperlink{main_theorem_a}{Main Theorems A} and \hyperlink{main_theorem_b}{B} for the cases C-4 / R-62 , GM-20 / K3-31 and GM-21 / K3-35 respectively. In \cref{Verra} we investigate an example related to EPW quartics
and give a short list of open problems and interesting questions.

\subsection*{Acknowledgements}
The authors wish to thank Sasha Kuznetsov for useful discussions and for suggesting a proof of \cref{prop:cbkuzpart}. We are also grateful to F. Catanese, B. Van Geemen and K.G. O'Grady for useful discussions.

GK is supported by the project Narodowe Centrum Nauki 2018/30/E/ST1/00530. MK is supported by the project Narodowe Centrum Nauki 2018/31/B/ST1/02857. 
MB and LM acknowledge support from the ANR project \emph{FanoHK}, grantANR-20-CE40-0023. EF, GM, and FT are supported by PRIN2020 research
grant ``2020KKWT53” and by PRIN2022 research grant ``2022PEKYBJ”, and are members of INdAM GNSAGA. FT is partially supported by the MIUR Excellence Department Project awarded to Dipartimento di Matematica, Università di Genova, CUP D33C23001110001, and by the Curiosity Driven 2021 Project \emph{Varieties with trivial or negative canonical bundle and the birational geometry of moduli spaces of curves: a constructive approach} - PNR DM 737/2021. 
Funded by the European Union - NextGenerationEU under the National Recovery and Resilience Plan (PNRR) - Mission 4 Education and research - Component 2 From research to business - Investment 1.1 Notice Prin 2022 - DD N. 104 del 2/2/2022, from title "Symplectic varieties: their interplay with Fano manifolds and derived categories", proposal code 2022PEKYBJ – CUP J53D23003840006.

\section{Generalities on conic bundles}
\label{generalitiesOnConic}

\subsection{Hyperbolic reduction}
Given a base scheme $B$, a quadric bundle $X \to B$ of relative dimension $r$ is given by the following data: a line bundle $\cL$ on $B$, a vector bundle $\cE$ of rank $r+2$ on $B$, and an injective morphism 
of vector bundles 
$q:\cL^\vee \to S^2 \cE^\vee$.
Equivalently, $q$ can be seen as a surjective morphism $S^2 \cE \to \cL$ or as a symmetric map 
$\cE\otimes \cL^\vee \to \cE^\vee$, everywhere non zero.
Then $X \subset \PP E$ is the zero locus of $q$.
Following \cite{auel-berna-bolo} we  denote this data by $(\cE,q,\cL)$.

Let $(\cG,q,\cL)$ be a quadric bundle of relative dimension $r+2$, so that $\cG$ has rank $r+4$.
If $\cN \subset \cG$ is a $q$-isotropic subbundle, and $\cN^\perp \subset \cG$ its orthogonal complement (which is locally free, see \cite{auel-berna-bolo}), we can consider the bundle $\cE=\cN^\perp/\cN$ and the induced map $q' : \cE \otimes \cL^\vee \to \cE^\vee$. This defines a quadric bundle $Y\to B$ of relative dimension $r$, with the same discriminant, Clifford algebra, etc.
We call $(\cE,q',\cL)$ the hyperbolic reduction of $(\cG,q,\cL)$ via $\cN$.

Recall from \cite[page 260]{auel-berna-bolo} that there is an exact sequence 
\[
0 \longrightarrow \cN^\perp/\cN \longrightarrow \cG/\cN \longrightarrow \mathcal{H}om(\cN,\cL) \longrightarrow 0,
\]
whose dual, since $\cE=\cN^\perp/\cN$, gives the  sequence
\begin{equation}\label{eq:useful-for-hyp-red}
0 \longrightarrow \cL^\vee \otimes \cN \longrightarrow (\cG/\cN)^\vee \longrightarrow \cE^\vee \longrightarrow 0.
\end{equation}

In the sequel we will try to compare pairs of conic bundles $X\subset\PP(\cE)$ and 
$X'\subset\PP(\cE')$ over the same threefold $B$. For this 
we will exhibit a quadric threefold bundle $Y\subset \PP(\cG)$ from which 
both $X$ and $X'$ can be derived by hyperbolic reduction, say
through isotropic line bundles 
$\cN$ and $\cN'$ of $\cG$. In particular the discriminant
surfaces in $B$ will be the same, and $X$ and $X'$ will be $K(B)$-birationally equivalent \cite{kuznethyperbolic}. 
Of course this is possible only when some special conditions 
are fulfilled. We record a simple observation. 

\begin{lemma}\label{section}
In the previous situation, let $D\subset B$ be the hypersurface over which $\cN$ and $\cN'$ 
are orthogonal in $\cG$. The conic bundles $X$ and $X'$ admit a section over $D$. 
\end{lemma} 

\proof By construction, $\cE\simeq \cN^\perp/\cN$. Over $D$, $\cN'$ being orthogonal to $\cN$ injects inside $\cE$ and provides the required section. \qed

\subsection{A general fact about conic bundles over threefolds}

Let $f: X \to B$ be a standard conic bundle over a smooth threefold $B$, that is, a triple $(\cE,q,\cL)$ as above, with $\cE$ a rank 3 vector bundle. The discriminant divisor $\Delta \subset B$ parametrises singular conics, the singularities of $\Delta$ correspond to double lines, while smooth points of $\Delta$ parametrize conics consisting of two lines meeting in one point \cite[Prop. 1.2]{beauvilleprym}.

We assume $\Delta$ to have only a finite number of nodal singularities, and to have a smooth double cover $\sigma: F \to \Delta$ ramified along the singularities. 
In this case, $F$ is at the same time the moduli space of $f$-lines in $X$, and a $0$-dimensional quadric fibration over $\Delta$. Note that $F$ comes naturally with an involution $\iota$, and we can consider the $\iota$-anti-invariant part $H^2(F)^-$ of $H^2(F,\ZZ)$.

Inspired from \cite{beauvilleprym} and by \cite{kuzmail}, we describe an isometry between $H^2(F)^-$ and the primitive part of $H^4(X,\ZZ)$, that we define as follows:
\[
H^4_{pr}(X,\ZZ) := (f^*H^4(B,\ZZ) \oplus f^*H^2(B,\ZZ) \cdot h)^{\perp},
\]
where $h$ is the restriction to $X$ of the relative hyperplane class of $\PP(\cE) \to B$ (we note that $h$ coincides with the anti-canonical divisor $-K_X$ up to twist with elements from the base $B$).

\begin{prop}\label{prop:cbkuzpart}
In the above setting, there is a morphism 
\[
\phi: H^2(F,\ZZ) \longrightarrow H^4(X,\ZZ), 
\]
preserving the Hodge structures, and such that
\[
\phi(a) \cdot \phi(b) = a \cdot (\iota^*b-b).
\]
\end{prop}

\begin{proof}
The proof of the first part goes along the lines of Beauville's proof of Th\'eor\`eme 2.1 in \cite{beauvilleprym} opportunely modified to take into account singularities of the discriminant divisor \cite{kuzmail}.

So we let $Z$ be the universal $f$-line in $X$, which comes as a correspondence:
\[\xymatrix{
 & Z \ar[dl]_\zeta \ar[dr]^p & \\
 X\ar[dr]_f & & F\ar[dl]_\sigma \\
  & B\supset \Delta &}  
\]
where $p$ is a $\PP^1$-fibration and $\zeta$ is the normalization of the divisor $f^{-1}(\Delta)$. Indeed, we can decompose $\zeta$ via the blow-up of $X$ along a resolution of singularities of $\Delta$ as follows.

We let $\pi_\Delta : \tilde{\Delta} \to \Delta$ be the blow-up of the singular locus, and $\tilde{\sigma}:\tilde{F} \to \tilde{\Delta}$ the 2:1 cover ramified along the exceptional divisor of $\pi_\Delta$. Equivalently, $\pi_F: \tilde{F} \to F$ is the blow-up along the ramification locus of $\sigma$.
Considering in $X$ the subscheme of singular fibres with a marked point allows one to define an embedding $\delta: \tilde{\Delta} \to X$.

We finally consider the blow-up $\pi:X' \to X$ of $X$ along $\tilde{\Delta}$ and note that $\zeta$ factors via an embedding $i:Z \to X'$. We end up with the following diagram:
%\giovanni{The $\zeta$ arrow has a bent tip, can someone fix it?}
%\fabio{done}
\[
\xymatrix{
\tilde{F} \ar@/_2pc/[dd]_{\tilde{\sigma}} \ar[d]^{s_E} \ar[r]_{s_Z} \ar@/^2pc/[rr]^{\pi_F} & Z \ar[d]_i \ar[r]^p \ar@/^2.01pc/[dd]^{\zeta} & F \\
E \ar[d]^q \ar[r]^{\varepsilon} & X' \ar[d]_\pi & \\
\tilde{\Delta} \ar[r]^\delta & X, &
}
\]
where $\varepsilon: E \to X'$ is the exceptional divisor of $\pi$.

We define $\phi:=\zeta_* p^*= \pi_* i_* p^*$, and we continue along the same lines of \cite{beauvilleprym}.
Indeed, since $X' \to X$ is the blow-up of a codimension $2$ smooth locus, we have:
\[
H^4(X',\ZZ)=\pi^*H^4(X,\ZZ) \oplus \varepsilon_* q^* H^2(\tilde{\Delta},\ZZ), 
\]
so that for any $y$ in $H^4(X',\ZZ)$ one can write
\[
y = \pi^* \pi_* y - \varepsilon_* q^* q_* \varepsilon^* y.
\]

Consider $a$ in $H^2(F,\ZZ)$, and its image $y=i_*p^* a$ in $H^4(X',\ZZ)$. The latter can be written, using the above diagram, as:

\[
\begin{array}{rl}
i_*p^* a &= \pi^*\pi_* i_* p^* a - \varepsilon_* q^* q_* \varepsilon^* i_* p^* a \\
&= \pi^* \zeta_*p^* a - \varepsilon_* q^* q_* s_{E*} s_Z^* p^* a \\
&= \pi^* \phi(a) - \varepsilon_* q^* \tilde{\sigma}_*\pi_F^* a.
\end{array}
\]
We calculate now $p_*i^* i_* p^* a$ by applying $p_* i^*$ to the above:
\[
\begin{array}{rl}
   p_*i^* i_* p^* a  &=  p_*i^*\pi^* \phi(a) - p_*i^*\varepsilon_* q^* \tilde{\sigma}_*\pi_F^* a  \\
   &= p_* \zeta^* \phi(a) - p_* s_{Z*} s_E^* q^* \tilde{\sigma}^* \pi_F^* a \\
   &= \phi^t \phi(a) - \pi_{F*} \tilde{\sigma}^* \tilde{\sigma}_* \pi_F^* a,
\end{array}
\]
where we denote by $\phi^t:=p_* \xi^*$ the adjoint map to $\phi$.

Remark that $i: Z \to X'$ is a divisorial embedding, so that $i^*i_*$ is nothing but intersecting with the first Chern class of the normal bundle of $Z$ in $X'$. On the one hand, we have $K_{X'}= \pi^* K_X + E$ by the blow-up formula. On the other hand, such canonical class restricts to a trivial divisor on the general fibre of the $\PP^1$-bundle $p: Z \to F$. This gives that the normal bundle of $Z$ in $X'$ has degree $-2$ on the general fibre of $p: Z \to F$. Therefore,

\[
\begin{array}{rl}
    p_* i^* i_* p^* a &= p_* (p^* a \cdot c_1(N_{Z/X'})  \\
    &= a \cdot p_* c_1(N_{Z/X'})\\
    &=-2a. 
\end{array}
\]
We deduce that 
\[
\phi^t \phi(a) = \pi_{F*} \tilde{\sigma}^* \tilde{\sigma}_* \pi_F^* a - 2a.
\]
Notice that $\tilde{\sigma}$ being a double cover with involution $\tilde{\iota}$, the first term of the right hand side of the above equality can be rewritten as
\[
\pi_{F*} \tilde{\sigma}^* \tilde{\sigma}_* \pi_F^* a = \pi_{F*}(\pi_F^* a + \tilde{\iota}^* \pi_F^* a) = a + \iota^* a.
\]
We can now write:
\[
\phi(a) \cdot \phi(b)= a \cdot \phi^t\phi(b) = a \cdot (\pi_{F*}\tilde{\sigma}^* \tilde{\sigma}_* \pi_F^* b - 2b)=a \cdot (\iota^*b-b),
\]
and this concludes the proof.\end{proof}

Now remark that if $a$ is in $H^4(B,\ZZ)$, and  $j: \Delta \to B$ denotes the inclusion,  
$$
\phi^t(f^*a)= p_*\xi^*f^* a = \sigma^* j^* a = 0,
$$
so that $\phi^t$ vanishes on $f^* H^4(B,\ZZ)$.
If $b$ is in $H^2(B,\ZZ)$, then
$$
\phi^t((f^*b).h)=p_* \xi^* ((f^*b).h)=
 p_*((p^* \sigma^* j^* b).(\xi^* h))= (\sigma^* j^* b).(p_*\xi^* h),
$$
where the last step is given by the projection formula. Since $p_* \xi^* = {\pi_F}_* \delta^* \tilde{\sigma}^*$, it follows that $\phi^t$ sends $f^* H^2(B,\ZZ).h$ to the $\iota$-invariant part of $H^2(F,\ZZ)$.

\begin{prop}\label{prop:cbisometry}
Up to multiplication by $-2$, the morphism $\phi$ induces an isometry between $H^2(F)^-$ and $H^4_{pr}(X,\ZZ)$.
\end{prop}

\begin{proof}
First of all, by \cref{prop:cbkuzpart}, the map $\phi$ satisfies \[
\phi(a) \cdot \phi(b) = -2(a \cdot b), \text{ for any } a,b \in H^2(F)^-.
\]
Since the intersection pairing on $H^2(F)^-$ is non-degenerate, the restriction of $\phi$ is injective. 

Consider the projection $pr: H^4(X,\ZZ) \to H^4_{pr}(X,\ZZ)$, and the composition $\psi:=pr \circ \phi$. The kernel $K$ of $pr$ is generated by $f^*H^4(B,\ZZ)$ and $f^*H^2(B,\ZZ) \cdot h$,
and by the considerations above we have that $K$ lies in $\phi(H^2(F)^+)$. Hence $K \cap \phi(H^2(F)^-) = 0$, so that the restricted composition $\psi_{\vert H^2(F)^-}$ is injective.

We set $V:=X \setminus f^{-1}(\Delta)$. Recalling that $\tilde{\Delta}$ is contained in $f^{-1}(\Delta)$, we easily see that
\[V=X \setminus f^{-1}(\Delta)=X' \setminus (E \cup Z)\]
The Leray spectral sequence for the inclusion of $j: V \to X'$ reads
\[
E^{p,q}_2 = H^p(X',R^qj_* \ZZ) \Rightarrow H^{p+q}(V).
\]
When restricted to the primitive part of the cohomology, we have that $H^4_{pr}(V)=0$, so that restricting the above spectral sequence to the primitive cohomology yields $E^{4,0}_{pr,\infty}=0$. 
Moreover, $d_2 : E^{2,1}_2 \to E^{4,0}_2$  identifies with the Gysin morphism
\[
(i_*,\varepsilon_*) : H^2(Z,\ZZ) \oplus H^2(E,\ZZ) \longrightarrow H^4(X',\ZZ)
\]
and a standard argument in Hodge theory (\cite[3.2.13]{deligne}) gives that the $d_r$ are trivial for $n \geq 3$, since $V$ is the complement of a normal crossing divisor. Hence $d_2$ is surjective on the primitive part.
We end up with a surjective map
\[
\gamma: H^2(Z,\ZZ) \oplus H^2(E,\ZZ) \xrightarrow{i_*,\varepsilon_*} H^4(X',\ZZ) \xrightarrow{pr'} H^4_{pr}(X',\ZZ) \xrightarrow{\pi_*} H^4_{pr}(X,\ZZ).
\]
First, $\pi_*$ is surjective and $\pi_* \varepsilon_*=0$, so that $\gamma$ is trivial on $H^2(E,\ZZ)$ and hence surjective when restricted to $H^2(Z,\ZZ)$.
Second, $H^2(Z,\ZZ)=p^*H^2(F,\ZZ) \oplus p^* H^0(F,\ZZ).l$, where $l$ is the relative hyperplane of the projective bundle $Z \to F$. Since $\pi_* i_* {s_Z}_*= \pi_* \varepsilon_* {s_E}_* =0$, we deduce that $\gamma$ is also trivial on $p^*H^0(F,\ZZ).l$. Therefore
\[\psi=\gamma \circ p^*: H^2(F,\ZZ) \to H^4_{pr}(X',\ZZ)
\]
is surjective.

Finally, the kernel of $\psi$ contains $H^2(F)^+$, but it does not contain any non-trivial anti-invariant element. Hence, $H^2(F)^+$ coincides with the kernel of $\psi$, and the claim follows.
\end{proof}

\section{16-nodal quintics and cubic fourfolds} 
\label{sect16nodal}

\subsection{Blow-up of a cubic along a line and nodal quintics}
Let $X_3 \subset \PP^5$ be a smooth cubic fourfold, and $\ell \subset X$ a general line. The projection $X_3 \dashrightarrow \PP^3$ from $\ell$ is a (rational) conic bundle structure on $X_3$. If $X \to X_3$ is the blow-up along $\ell$, we obtain a conic bundle $X \to \PP^3$ whose discriminant $\Delta$ is a quintic surface with $16$ nodes. Such quintics and their discriminant double covers have many interesting features. 
    
First of all, given any quintic with an  even set of $16$ nodes one can present it as $F/\iota$, 
where $\iota$ is an involution of a surface $F$ of general type. According to Beauville \cite{beauville} there is a unique irreducible family of such quintics.

On the other hand, for $\ell\subset X_3$ a line in a cubic fourfold, the set of lines meeting $\ell$ is in general a smooth surface $F$ of general type,  with a natural involution $\iota$ whose quotient is the $16$-nodal quintic $\Delta$ given by the discriminant of the conic bundle $X \to \PP^3$. As shown by Catanese, \cite{ca}, any $16$-nodal quintic can be obtained this way. 

In \cite{hu21} the interplay between $F \to \Delta$, the K3 structure of $X_3$ and the \HK{} Fano variety of lines $F(X_3)$ are extensively studied. Assertions  (5) and (6) from \hyperlink{main_theorem_b}{Main Theorem B} are proved there for  this particular case.

\begin{prop}[{\cite{hu21}}]\label{prop:huy-for-C4}
Let $X \to \PP^3$ be the conic bundle given by the blow-up of a cubic fourfold $X_3$ along a line, and $F \to \Delta$ the discriminant double cover.
 
There is an isometry of Hodge structures between the anti-invariant part of $H^2(F)^{-}$ and the primitive part of $H^4(X,\ZZ)$, and, henceforth with $H^2(F(X_3),\ZZ)$ endowed with the Beauville--Bogomolov form.  In particular $H^2(F)^-$ has a K3 Hodge structure of signature $(2,21)$.

The surface $F$ is smooth, minimal, of general type,  and algebraically simply connected. 
Moreover:
$$\chi(\cO_F)=6, \quad  q(F)=0, \quad p_g(F)=5, \quad e(F)=62, \quad  b_2(F)=60,\quad c^2_1(F)=10.$$
\end{prop}

Note that the isometry between $H^2(F)^-$ and $H^4_{pr}(X,\ZZ)$ can also be proved as a consequence of \cref{prop:cbisometry}.

The variety $X$ appears in the list of Fano fourfolds of K3 type obtained in \cite{bfmt}, as item C-4: it is the zero locus in $ \PP^3 \times \PP^5$ of a general section of  $\cO(1,2) \oplus \cQ_{\PP^3}(0,1)$. Indeed the projection to $\PP^5$ is the blow-up of a cubic fourfold $X_3$ along a line $l$, and the projection to $\PP^3$ is the residual conic bundle. 

\subsection{Another conic bundle with quintic discriminant}
Let $X'$ in $\PP^3 \times G(2,7)$ be the zero locus of a general section of $ \cO(0,1)^ {\oplus 2} \oplus \cO(1,1) \oplus \cQ_{\PP^3}\boxtimes \cU^{\vee}_{G(2,7)}$. As shown in \cite{bfmt}, case R-62, the projection $\pi_1 : X' \to \PP^3$ is also a 
conic bundle whose discriminant $\Delta'$ is a $16$-nodal quintic. 
The projection $X' \to \Gr(2,7)$ 
is birational, contracts exactly a plane and so its image is a singular fourfold, that can be seen as a $(2,2)$-divisor in $\PP^2 \times \PP^3$.

Consider the discriminant double cover $F' \to \Delta'$, then $F'$ is a variety of lines in the fibres of $\pi_1$. By Catanese's results on nodal quintics and their double covers, $F' \to \Delta'$ must be isomorphic to some $F \to \Delta$ from a conic bundle of type C-4. We provide an explicit construction of this identification, showing how the family R-62 can be interpreted as a special subfamily of C-4.

Because of the 
factor $\cQ_{\PP^3}\boxtimes \cU^{\vee}_{\Gr(2,7)}$, we are reduced to a $\Gr(2,4)$-bundle 
over $\PP^3$, namely $\Gr(2,\cO(-1)\oplus \cO^{\oplus 3})$. Lines in the $\Gr(2,4)$ fibres are parametrised
by points in $\Fl(1,3,4)$, so we consider the fibre bundle $\Fl(1,3,\cO(-1)\oplus \cO^{\oplus 3})$ and its tautological bundles $U_1\subset U_3$. We get the following.

\begin{prop}
The variety $F'$ of $\pi_1$-lines in $X'$ in the flag bundle over $\PP^3$ $\Fl(1,3,\cO(-1)\oplus \cO^{\oplus 3})$  is the zero locus of a general section of the rank six vector bundle $(U_1\wedge U_3)^\vee\otimes (\cO(1)\oplus \cO^{\oplus 2})$. It is a smooth minimal surface of general 
type, admitting an involution whose quotient is the discriminant 16-nodal quintic $\Delta'$. 
\end{prop}

\begin{proof}
There only remains to prove that $F'$ is minimal, but it is enough to remark that its canonical divisor is ample, as a double cover of a quintic surface ramified in $16$ nodes \cite{hu21}.
\end{proof}

Note that $U_1\wedge U_3\simeq U_1\otimes (U_3/U_1)$ is an irreducible homogeneous rank two  vector bundle over the flag manifold. Using the Bott--Borel--Weil theorem we can then compute cohomologies with the aid of \cite{M2}, and deduce following the numerical invariants of $F'$.

\begin{prop}\label{prop:inv-of-F-for-cubic}
%The surface $F'$ is a minimal surface of general type. 
The numerical invariants of $F'$ are 
$$\chi(\cO_{F'})=6, \quad e(F')=62, \quad K_{F'}^2=10, \quad \chi(-K_{F'})=16.$$
\end{prop}

\begin{corollary}\label{cor:F-Simply-conn-C4}
The surface $F'$ is algebraically simply connected.
\end{corollary}

\begin{proof}
Since $3K_{F'}^2 < 8(\chi(\cO_F')-2)$, this is a consequence of \cite[Cor.\ 4]{Xi87}.
\end{proof}

Note that \cref{prop:inv-of-F-for-cubic,cor:F-Simply-conn-C4} are in agreement with \cref{prop:huy-for-C4} and the expected isomorphism between $F'$ and (some) $F$.

%\marcello{The next remark should go to the open problem section?}
%\laurent{I would erase it}
%\enrico{commented}
%It would be interesting to see how these surfaces deform, in order to understand why
%$\chi(T)=-23$. At the deformations obstructed? Does $\chi(T_F)=-23$?
%As a matter of fact, $\chi(T_F)=-40$.

%Of course one would expect that R-62 be related to some special type of cubics, 
%but which ones? 
%We would understand this case better if we had some understanding of the geometry
%of the Fano fourfolds from this family. Can we flip the contraction to the singular model
%in $G(2,7)$? 

\subsection{Comparing the conic bundles}
\subsubsection{The  conic bundle structure on C-4}
Recall that the conic bundle from C-4 is just obtained by projecting a cubic hypersurface 
in $\PP^5$ to $\PP^3$ from a line contained in the cubic. The corresponding rank three 
bundle is  $\cE = \cO_{\PP^3}^{\oplus 2} \oplus \cO_{\PP^3}(-1)$ 
and the quadratic form is a morphism
$$q : \cO_{\PP^3}(-1) \to \mathrm{Sym}^2 \cE^\vee.$$

Conversely, let $\xi$ be the divisor class of $\cO_{P(E)}(1)$ and $h$ be the pullback of the hyperplane class from $\PP^3$ to $\PP(E)$. Note that $\xi$ is then the pullback of the hyperplane class on $\PP^5$ by the blow up and the class of the exceptional locus is $\xi-h$. Then the conic bundle associated to a map $q$ as above is a divisor on $\PP(E)$ of class $2\xi+h$, adding the class of the exceptional divisor we get an effective divisor in the class $3\xi$ hence a pre-image by the blow up  of a cubic fourfold, that necessarily contains the blown up line. We conclude that every conic bundle associated to a map $q$ as above arises as C-4. In fact one can see the relation between cubics containing a line and the maps $q$ explicitly. If we we fix a coordinate chart such that the line is defined by the vanishing of four coordinates in $\PP^5$ then the corresponding matrix is the matrix of coefficients of cubic in front of the monomials in the remaining two coordinates. More precisely, choose coordinates $x_1\dots x_6$ on $\mathbb P^5$  such that the line that we blow up is given by equations $x_3=x_4=x_5=x_6=0$. Then cubics containing the line are given by equations of the form: 
$c+x_1q_1+ x_2q_2+x_1^2 l_1+x_1x_2 l_2+x_2^2 l_3=0$,
where $c$, $q_1$, $q_2$, $l_1$, $l_2$, $l_3$ are polynomials of respective degrees $3$, $2$, $2$, $1$, $1$, $1$. The map $q$ is given by the same data
as the following 3x3 matrix of polynomials in $x_3,x_4,x_5,x_6$ representing the coordinates in $\PP^3$ 
$$\left(
\begin{array}{ccc}
 c  & q_1 & q_2 \\
    q_1 & l_1& l_2 \\
    q_2 & l_2 & l_3
\end{array}
\right).
$$

\subsubsection{The  conic bundle structure on R-62}
On the other hand, as explained in \cite{bfmt} the conic bundle R-62 lives inside $\PP(\cF)$ where the rank three vector bundle $\cF$ is defined  by an exact sequence (with a  fixed three-dimensional vector space $K_3$)
$$0 \longrightarrow \cO^{\oplus 2} \oplus \cO(-1) \stackrel{\alpha}{\longrightarrow} \wedge^2((K_3 \otimes \cO)\oplus \cO(-1))^\vee \longrightarrow \cF^\vee \longrightarrow 0,$$
where the maps are general.

The middle term can then be rewritten as  $\cO^{\oplus 3} \oplus \cO(1)^{\oplus 3}$. 
Since  the map $\alpha$, restricted to $\cO^{\oplus 2}$, is in general given by constants, we can rewrite this sequence as 
$$0 \longrightarrow \cO(-1) \longrightarrow \cO \oplus \cO(1)^{\oplus 3} \longrightarrow \cF^\vee \longrightarrow 0.$$
The conic bundle structure in $\PP(\cF)$ arises from a general  morphism $\mathrm{Sym}^2 \cF \to \cO(-1)$. We will rather twist $\cF$ to get $\cE'=\cF(1)$, with an embedding
$$q': \cO(-1) \longrightarrow \mathrm{Sym}^2 \cE'^\vee.$$
Note that there is an isomorphism
$$H^0(\PP^3, \mathrm{Sym}^2 \cE'^\vee(1))\simeq K_3^\vee\otimes (\wedge^2K_3^\vee/\langle \omega_1, \omega_2\rangle)\oplus S^2K_3^\vee\otimes V_4^\vee , $$
in particular there is a family of dimension $26$ of conic bundles 
on $\PP(\cE')$. The one we consider is defined by the image of the line $\wedge^3K_3^\vee$ in $K_3^\vee\otimes\wedge^2K_3^\vee/\langle \omega_1, 
\omega_2\rangle$, and it is a priori unclear that it is general. 
On the other hand $H^1(\PP^3, \mathrm{End}(\cE'))$ 
is only twelve-dimensional, and we can deform the triple $(\omega_1,\omega_2,\beta)$ that defines R-62 without changing $\cE'$, but with
the effect of deforming our conic bundle non-trivially. We  need to check that 
in the end,  we get the generic such conic  bundle.

%\michal{I added general in several places above, is this correct? More precisely, does it follow from [BFMT21] directly that any conic bundle as here is of type R-62, or do we need an additional argument?}
%\laurent{I included an awkward explicit argument.}

\begin{lemma}
The generic bundle $\cE'$ as above, with a general conic bundle structure defined 
by $q': \cO(-1) \longrightarrow \mathrm{Sym}^2 \cE'^\vee,$ can be obtained from R-62.
\end{lemma}

\proof 
Let us compare the two situations. Given $\cE'$, consider the bundle $\cG=\cE'(-1)$, defined by a 
sequence 
$$0 \longrightarrow \cG \longrightarrow \cO \oplus \cO(-1)^{\oplus 3} 
\stackrel{g}{\longrightarrow} \cO(1) \longrightarrow 0,$$
where the generic $g=(a,q_1,q_2,q_3)$ is given by a linear form $a$ and three quadratic 
forms $q_1,q_2,q_3$. Any morphism $q : \mathrm{Sym}^2\cG\ra \cO(-1)$ is the restriction 
of a morphism  $Q : \mathrm{Sym}^2(\cO(1) \oplus \cO^{\oplus 3})\ra \cO(1)$, given by some 
symmetric matrix 
$$\begin{pmatrix} 0 & \alpha_1& \alpha_2& \alpha_3 \\
 \alpha_1 & \nu_{11} & \nu_{12} & \nu_{13} \\ 
 \alpha_2 & \nu_{12} & \nu_{22} & \nu_{23} \\ 
 \alpha_3 & \nu_{13} & \nu_{23} & \nu_{33} 
\end{pmatrix}$$
where the $\alpha_i$'s are scalars and the $\nu_{jk}$ are linear forms. 
Over the open subset where $a\ne 0$, $\cG$ is generated by the 
vectors $g_i=q_i(\ell)e_0-a(\ell)\ell \otimes e_i$, $1\le i\le 3$, and 
$$q(g_i,g_j) = a(\ell)\Big(a(\ell)\nu_{ij}(\ell)
-\alpha_iq_j(\ell)-\alpha_jq_i(\ell)\Big).$$
We can even normalize further by supposing that $(\alpha_1,\alpha_2,\alpha_3)=(0,0,-1)$,
which gives a matrix for $a(\ell)^{-1}q$ of the form 
\begin{equation}\label{gen}
a\begin{pmatrix} \nu_{11} & \nu_{12} & \nu_{13} \\ 
  \nu_{12} & \nu_{22} & \nu_{23} \\ 
 \nu_{13} & \nu_{23} & \nu_{33} 
\end{pmatrix}+ \begin{pmatrix} 0& 0 & q_1 \\ 0 & 0 &q_2 \\ q_1& q_2 & 2q_3 
\end{pmatrix}.
\end{equation}
For a bundle $\cF$ coming from R-62 as above, the initial data are two skew-symmetric forms 
$\omega_1, \omega_2$ on $V_7$, and $\beta\in\wedge^2V_7^\vee\otimes V_4^\vee$,
a skew-symmetric form with values in $V_4^\vee$. We normalize the restrictions 
of $\omega_1, \omega_2$ to $K_3$ by fixing a basis $(e_1,e_2,e_3)$ such that 
$$\begin{array}{ccc}
   \omega_1(e_1,e_2)=0,   &  \omega_1(e_1,e_3)=1, &  \omega_1(e_2,e_3)=0, \\
  \omega_2(e_1,e_2)=0,   & \omega_2(e_1,e_3)=0, &  \omega_2(e_2,e_3)=1.
\end{array} $$
Denote by $\beta_{ij}$ the linear form $\beta(e_i,e_j)$. Over the open set 
where $\beta_{12}\ne 0$, $\cF$ is generated by the bivectors of the form
\begin{multline*}
\beta_{12}(\ell)\Big(k\wedge\ell-\omega_1(k,\ell)e_1\wedge e_3
-\omega_2(k,\ell)e_2\wedge e_3\Big)-\\\Big(\beta(k,\ell)-\omega_1(k,\ell)\beta_{13}
-\omega_2(k,\ell)\beta_{23}\Big)(\ell)e_1\wedge e_2,
\end{multline*}
for $k$ in $K_3$, hence by the bivectors $f_i=\beta_{12}(\ell)e_i\wedge\ell+\gamma_i$, for some
$\gamma_i\in\wedge^2K_3$, and $1\le i\le 3$. The quadratic form evaluated on these 
vectors gives 
$$q(f_i,f_j)=\beta_{12}(\ell)(e_i\wedge \gamma_j+e_j\wedge \gamma_i)\wedge\ell.$$
Since $\gamma_i = \beta_{12}(\ell)(\psi_i(\ell)e_2\wedge e_3-\phi_i(\ell)e_1\wedge e_3)+\kappa_i(\ell)e_1\wedge e_2$, and we deduce that
the matrix of $\beta_{12}(\ell)^{-1}q$ is of the form
\begin{equation}\label{gen+}\beta_{12}\begin{pmatrix} \psi_1 & \phi_1+\psi_2 & \psi_{3} \\ 
  \phi_{1}+\psi_2 & \phi_{2} & \phi_{3} \\ 
 \psi_{3} & \phi_{3} & 0 
\end{pmatrix}+ \begin{pmatrix} 0& 0 & \kappa_1 \\ 0 & 0 &\kappa_2 \\ \kappa_1& 
\kappa_2 & 2\kappa_3 
\end{pmatrix}.
\end{equation}
Comparing the forms of \cref{gen} and \cref{gen+}, we can conclude that 
R-62 allows us to recover the generic $\cE'$ and its generic conic bundle structure 
if we can suppose that the linear forms 
$\phi_i, \psi_j$ and the quadratic forms $\kappa_k$ are generic. But this is 
clear, because
$$\phi_i(\ell)=\omega_1(e_i,\ell),\!\! \quad \psi_j(\ell)=\omega_2(e_j,\ell),\!\! \quad 
\kappa_k(\ell) = \beta(e_k,\ell)(\ell)-\phi_k(\ell)\beta_{13}(\ell)
-\psi_k(\ell)\beta_{23}(\ell),$$
and $\omega_1, \omega_2, \beta$ are generic. \qed

\subsubsection{How to connect them through hyperbolic reduction}\label{conic bundle on R-62}
Consider on $\PP^3$ the rank five vector bundle 
$$\cG= \cO(1) \oplus \cO^{\oplus 3} \oplus \cO(-1),$$
with a quadratic form
$$Q: \cO(-1) \to \mathrm{Sym^2}\cG^\vee, $$
and $Y \to \PP^3$ the induced quadric bundle.

%The Chern classes calculation would indeed give a degree 5 discriminant (note that here the rank of $\cG$ is five so we have $\Delta = 5 c_1( \cO(-1)) + 2 c_1 (\cG^\vee)$).

The subbundle $\cN=\cO(1) \subset \cG$ is always isotropic. Moreover, $\cO^\perp\simeq \cO(1) \oplus \cO^{\oplus 2} \oplus \cO(-1)$, so that  hyperbolic reduction of $\cG$ via $\cN$ 
yields a conic bundle $X$ of type C-4.

If we ask for an additional subbundle $\cN'\simeq \cO(-1) \subset \cG$ to be isotropic, the sequence 
\cref{eq:useful-for-hyp-red} shows  that hyperbolic reduction yields a conic bundle $X'$ of type R-62. 
But this is an extra condition on $Q$ that will be reflected in the geometry of $X$. 
Applying \cref{section}, we get:

\begin{prop} \label{prop:C4-R62-hypred} For a general $X'$ of type R-62, there exists $X$ of type C-4 and a quadric threefold bundle $Y\to \PP^3$ such that $X$ and $X'$
are obtained by hyperbolic reduction from $Y$. In that case, $X$ corresponds to a cubic fourfold containing a plane that does not meet the blown-up line. 

Conversely, if $X$ is the blow up of a cubic fourfold containing a plane that does not meet the blown up line, then there exists $X'$ of type R-62 and a quadric threefold bundle $Y\to \PP^3$ such that $X$ and $X'$
are obtained by hyperbolic reduction from $Y$. 
\end{prop} 

\begin{proof} Consider $X'$ of type R-62. By the discussion in \cref{conic bundle on R-62}, $X'$ appears as a symmetric degeneracy locus associated to $q': \cO(-1)\to \mathrm{Sym}^2  \cE'^{\vee}$. The latter by \cite[Proposition 3.6]{BKK} appears as a hyperbolic reduction of a quadric surface bundle associated to some $Q:  \cO(-1)\to \mathrm{Sym}^2  \cG^{\vee}$ with respect to a subbundle $\mathcal O(-1)\subset \cG$. The latter quadric bundle $Q$ has another isotropic subbundle $\mathcal O(1)\subset \cG$ and the hyperbolic reduction with respect to that bundle gives a conic bundle associated to a map $q: \cO(-1)\to \mathrm{Sym}^2  \cE^{\vee}$. By \cref{section} there exists a plane in $\mathbb{P}^3$ over which $q$ admits a section. This section is then mapped by the blow up map to a plane contained in the cubic fourfold, that does not meet the blown up line.

Let us pass to the other direction. Assume that we have a cubic fourfold containing a plane. Choose coordinates $x_1\dots x_6$ on $\mathbb P^5$  such that the plane is given by equations $x_1=x_2=x_3=0$ and $x_3=x_4=x_5=x_6=0$ are equations of the blown up line. Then the cubic blown up in the line gives rise to the conic bundle $X$ defined by a matrix of the form:
$$\left(
\begin{array}{ccc}
 x_3  q_0  & q_1 & q_2 \\
    q_1 & l_1& l_2 \\
    q_2 & l_2 & l_3
\end{array}
\right).
$$
with $q_i$, $l_i$ forms of degrees 2 and 1 respectively in the variables $x_3,x_4,x_5,x_6$.
Indeed, equations of cubics containing the plane are of the form  $x_1f_1(x_1,\dots,x_6)+x_2f_2(x_2,\dots,x_6)+x_3q_0(x_3,\dots,x_6)=0$ and the matrix of the conic fibrations arises by writing down coefficients at the monomials in $x_1$, $x_2$ of the cubic equation.
Let us now consider the matrix 

$$\left(
\begin{array}{ccccc}
0& q_1 & q_2 & q_0 &-\frac{1}{2}x_3 \\
q_1 & l_1& l_2 & 0 & 0  \\
q_2 & l_2 & l_3& 0 & 0  \\
q_0 & 0   & 0  & 0 & 1  \\
-\frac{1}{2}x_3 & 0   & 0  & 1 & 0
\end{array}
\right),
$$
giving rise as a map $Q:  \cO(-1)\to \mathrm{Sym}^2  \cG^{\vee}$ to a quadric threefold bundle. Observe that this bundle has two isotropic sections given by $\cO(1)\subset \cG$ and  $\cO(-1)\subset \cG$ corresponding to the first and last row or column. Clearly, the quadric reduction with respect to the second section gives back $X$, whereas the quadric reduction with respect to the first section gives a conic bundle associated to a map $ \cO(-1)\to \mathrm{Sym}^2  \cE'^{\vee}$ hence is a variety of type R-62.
%\marcello{Write a couple of lines and think of the if and only if.}
%\michal{I reformulated the theorem to if and only if and added a proof, please check.}
%\laurent{Looks perfect}
\end{proof}

\begin{remark}
The Fano variety of lines for a  cubic containing a plane is a hyper-K\"ahler fourfold with Picard lattice $(8)\oplus (-2)$ such that the $-2$ class
is of divisibility $2$ and is represented by a divisor $E$ that is a $\PP^1$ fibration over a K3 surface of degree $2$.
 The discriminant surface $F$ is naturally embedded in $F(X)$. One can see that $E\cap F$ is isomorphic to a genus $2$ curve. Moreover $F$ has a structure of fibration induced from the Lagrangian fibration on $F(X)$. 
\end{remark}

\begin{corollary}\label{coro:HS-forC4-R62}
If $X$ and $X'$ are as in \cref{prop:C4-R62-hypred}, then $X$ and $X'$ are $K(\PP^3)$-birationally equivalent, and $F \simeq F'$.
Moreover, we have an isometry $H^4_{pr}(X) \simeq H^4_{pr}(X')$.
\end{corollary}

\begin{proof}
    The first part is a consequence of \cite[Prop.\ 1.1]{kuznethyperbolic}. The second is a consequence of \cref{prop:cbisometry}, via $H^2(F)^-$ and the isomorphism $F \simeq F'$ compatible with the involution.
\end{proof}

We thus have completed the proof of \hyperlink{main_theorem_b}{Main Theorem B} in this case.

%\marcello{To prove: we need to show that we get the right conic bundle (as discussed on zoom)}

Moreover, we remark that \cref{coro:HS-forC4-R62} gives an explanation for the K3 structure of $X'$, which was not explained in \cite{bfmt}, via the identification with cubic fourfolds containing a plane.

\section{40-Nodal sextics and Gushel--Mukai fourfolds}
\label{sect40nodal}

\subsection{Blow-up of GM along a degree 4 del Pezzo and nodal sextics}
Let $X_{10}$ be a general Gushel--Mukai fourfold. We first recall how to construct a conic bundle structure $X_{10} \dashrightarrow \PP^3$.
Inside the flag variety $\Fl(1,2,5)$, let $X$ be the zero locus of a general section of 
$\cO(1,0) \oplus \cO(0, 1) \oplus \cO(0,2)$. As described in \cite{bfmt}, item GM-20, we have that:

\begin{itemize}[leftmargin=*]
    \item the first projection $\pi_1 \colon X \to \PP^4$ is a conic bundle over a $\PP^3 \subset \PP^4$, whose discriminant is a sextic surface $\Delta$ with $40$ nodes.  
    \item the second projection $\pi_2 \colon X \to \Gr(2,5)$ is a blow-up $Bl_{dP_4}X_{10}$ along a del Pezzo surface of degree 4, where the Gushel--Mukai fourfold $X_{10}$ is general.
\end{itemize}
   
   The factor $\cO(1,0)$ means that we cut $\PP^4$ by some hyperplane $H$. So the construction
   can be rephrased as follows. We start with the fourfold $X_{10}$, and its variety
   of lines $F(X_{10})$, which is a smooth threefold. Each line in the Grassmannian has a
   pivot, which is the common point to the lines in the pencil. Then we define $F_H\subset 
   F(X_{10})$ as the surface of lines whose pivot is in $H$. This is a smooth surface in general, admitting a natural involution whose quotient is the nodal sextic $\Delta$. 
   
   More concretely, a line $\ell\subset H\subset V_5$ defines a four-dimensional 
   space $\ell\wedge V_5\subset \wedge^2V_5.$ A section of $\cO(0,1)$ is given by a linear form on the latter space, and by restriction we get a morphism $\cO(0,-1)\rightarrow 
   (\ell\wedge V_5)^\vee$ which in general degenerates in codimension $4-1+1=4$, hence nowhere. So the cokernel is a rank three bundle $\cE^\vee$. Finally, the section of $\cO(0,2)$
   is given by a quadratic form $q$ on  $ \wedge^2V_5$, and by restriction to $\cE$ we get 
   our conic bundle. 
   
   Lines in these conics correspond to hyperplanes in the fibres of $\cE$ on which $q$ vanishes identically. So we shall consider the fibre bundle $\Gr(2,\cE)$ over $\PP^3$, and the 
   discriminant double cover $F \to \Delta$ (parametrising vertical lines) is defined inside this bundle by a section of $S^2\cF^\vee$,
   where $\cF$ denotes the rank two tautological bundle of the relative Grassmannian.
   
   \begin{prop} 
   The surface $F$ is a minimal surface of general type.
The numerical invariants of $F$ are 
$$\chi(\cO_F)=12, \quad e(F)=96, \quad K_F^2=48, \quad \chi(-K_F)=60.$$
Moreover, $\chi(T_F)=-24$.
\end{prop} 

\begin{proof}
As a double cover of a sextic surface ramified in 16 nodes, the canonical divisor of $F$ is ample. The description of $F$ as the zero locus of a section of $S^2\cF^\vee$ allows us to compute its numerical invariants of $F$ with the aid of \cite{M2}.
\end{proof}

\begin{corollary}\label{cor:F-Simply-conn-GM20}
The surface $F$ is algebraically simply connected.
\end{corollary}

\begin{proof}
This follows again from the fact that $3K_{F}^2 < 8(\chi(\cO_F)-2)$, by \cite[Cor.\ 4]{Xi87}.
\end{proof}

\begin{comment}
restart
load "hodgeForProductsOfFlag.m2"
P4=flagBundle({1,4})
B=bundles P4
middle=dual(B_0 * B_1)
P3=sectionZeroLocus(dual B_0)
Edual=abstractSheaf(P3,ChernCharacter=>ch(middle*OO_P3)-ch(OO_P3))
E=dual Edual

Gr=flagBundle({2,1},E)
S=sectionZeroLocus(symmetricPower(2,dual first bundles Gr))
euler(S)
vol:=integral product apply(dim S,i->(chern(1,tangentBundle(S))))
h0:=integral(todd(tangentBundle(S))*ch(det(tangentBundle(S))))
chi(OO_S)
chi(tangentBundle(S))
\end{comment}

\begin{prop}
Up to multiplication by -2, there is a Hodge isometry between $H^2(F)^{-}$ and $H^4_{pr}(X,\ZZ)$. In particular, $H^2(F)^-$ has a K3-type Hodge structure with $h^{1,1}=25$.
\end{prop}

\begin{proof}
The isometry is given by \cref{prop:cbisometry}. Since $X \to X_{10}$ is a blow-up of a del Pezzo surface of degree 4, we have that $h^{2,2}(X)=27$. On the other hand $H^4_{pr}(X,\ZZ)$ is obtained as the orthogonal to two algebraic classes, so that $h^{2,2}_{pr}(X)=25$.
\end{proof}

We have thus proved \hyperlink{main_theorem_b}{Main Theorem B} (5) and (6) in this case.

\smallskip

The variety of lines $F(X_{10})$ in a smooth Gushel--Mukai fourfold $X_{10}$ was described in \cite[Theorem 4.7 (c)]{dk-linear}
   precisely in terms of the pivot map, denoted $\sigma : F(X_{10})\rightarrow \PP(V_5)$.
   The claim is that this map factors as 
   $$F(X_{10})\rightarrow \tilde{Y}_{A,V_5}\stackrel{f_{A,V_5}}{\rightarrow} Y_{A,V_5}\hookrightarrow
   \PP(V_5)$$
   where $Y_A$ is the EPW-sextic in $\PP^5$ associated to $X_{10}$ and $Y_{A,V_5}\subset
   \PP(V_5)$ is a hyperplane section; $\tilde{Y}_A$ is the double EPW-sextic with its 
   double covering $f_A : \tilde{Y}_A\rightarrow Y_A$, which restricts to $f_{A,V_5} : \tilde{Y}_{A,V_5}\rightarrow Y_{A,V_5}$. We deduce from this description the following 
   result, which is \hyperlink{main_theorem_a}{Main Theorem A} for this case.
   
   \begin{prop}\label{codim2}
   If $\PP(V_4)\subset \PP(V_5)$ is general, then the general type surface $F_{V_4}$ is $F$, and its double cover is given by
$$   F=\tilde{Y}_{A,V_4}\stackrel{f_{A,V_4}}{\rightarrow} Y_{A,V_4}=\Delta.$$
In particular $\Delta$ is just a codimension two linear section of the EPW-sextic $Y_A$.
   \end{prop}
   
   \begin{proof} We have seen that $F$ is the surface of lines in $X$ whose pivot is contained 
   in $\PP(V_4)$, that is $\sigma^{-1}(\PP(V_4))\subset F(X_{10})$. But by 
   \cite[Corollary 5.2]{dk-linear}, $F(X_{10})$ and $\tilde{Y}_{A,V_5}$ coincide outside
   the pre-images of the point  $[v_0]\in\PP(V_5)$ defined by the kernel of the two-form that 
   cuts out the linear section of $\Gr(2,V_5)$ containing $X_{10}$. In general $\PP(V_4)$ does not contain 
   $[v_0]$ and our claim follows. \end{proof}

%\marcello{The next two paragraph should either go to the last section or be erased}

%   In particular the surface $F$ is naturally embedded inside the double EPW-sextic 
%   $\tilde{Y}_A$. Does it have any special property with respect to the symplectic 
 %  structure?
 %\giovanni{As we obtain it by taking just two hyperplane sections (and no excess intersection) it should not be lagrangian by some computations of Markman (or more elementarily: the intersection $h^2\cdot \sigma \cdot \overline{\sigma}$ is positive)} %A strange thing is that we get a four-dimensional of surfaces, and any
%   two of them meet at least along a curve?? Also we would expect $20+8$ parameters for the   surface? 
%\laurent{So nothing special, I suggest to erase these remarks.}
%\giovanni{I've commented them}   

%   Note that sextic surfaces in $\PP^3$ have $83$ parameters, hence $68$ modulo $PGL_4$,
%   which reduces to $28$ when we impose $40$ ODP. Catanese and Casnati proved that on a sextic,  
%   an even set of nodes has cardinality $24, 32, 40$ or $56$. A natural question is whether the family of even $40$-nodal sextics is irreducible. There are  $24$ parameters
%   for GM-fourfolds, which give $28$ parameters if we add the choice of a hyperplane, so 
%   we could expect that the generic sextic with $40$ ODP comes from Gushel--Mukai's. 
%   But Catanese and Casnati only prove existence by exhibiting an example.

\subsection{Another conic bundle with sextic discriminant} Let $X'$ in $\PP^3 \times \PP^6$, be the zero locus of a general section of $\cO(0,2) \oplus \cO(1,1) \oplus \wedge^2 \cQ_{\PP^3}(0,1)$. As described in \cite{bfmt}, item K3-31, we have that:

%\subsubsection*{Invariants}  $h^0(-K)=23, \ (-K)^4=80, \  h^{1,1}=2, \ h^{3,1}=1, \ h^{2,2}=28$, $-\chi(T)=27$.

\begin{itemize}[leftmargin=*]
\item The first projection $\pi_1 \colon X' \to \PP^3 $ is  a conic bundle whose discriminant $\Delta'$ is a sextic hypersurface with $40$ nodes.
 \item $\pi_2\colon X' \to \PP^6$ is a blow-up $Bl_{S_{10}}(\QQ_1^5 \cap \QQ_2^5)$ of the intersection of two quadrics, for $S_{10}$ a K3 surface of genus 6.
\end{itemize}

The above description allows us to calculate the invariants of $F'$, that can be defined as the zero locus of a general section of some vector bundle, and show that they indeed coincide with the ones of $F$. We omit these calculations, since we will prove that $F'$ is isomorphic to some $F$ as in the the GM-20 case. %Indeed, the upshot is that any such $X' \to \PP^3$ is related to a special case of an $X \to \PP^3$ of type GM-20. In this case, there is no evidence coming from the moduli of nodal sextic to prove that the discriminant double cover $F' \to \Delta'$ must be isomorphic to some $F \to \Delta$, but we will obtain the isomorphism by hyperbolic reduction.

\subsection{Comparing the conic bundles}\label{CB}
%Similarly to the first case, we can try to think of the case K3-31 as a special case of GM-20 in the following sense: given the family of nodal sextics in $\PP^3$ obtained as degeneracy locus of GM-20, there is a codimension one subfamily that can appear also as a degeneracy locus of K3-31. The latter should be given by the existence of some vector bundle on the sextic.
\subsubsection{The  conic bundle structure on GM-20}
Recall that GM-20 is the family of Fano fourfolds $X$ defined as zero loci of general sections of 
\begin{equation}\tag{GM-20}
\cO(1,0) \oplus \cO(0,1) \oplus  \cO(0,2) \xrightarrow{} \Fl(1,2,5).\end{equation}
The projection to $\Gr(2,5)$ is the blow-up of a quartic del Pezzo surface in a general Gushel--Mukai fourfold. 
The conic bundle structure on $X$ is given by the projection to $\PP^4=\PP(V_5)$. 
Indeed,  $X$ parametrizes flags $(U_1\subset U_2\subset V_5)$. The $\cO(1,0)$ factor 
imposes that $U_1$ belongs to some hyperplane $V_4\subset V_5$. 
%\laurent{In 4.1 the hyperplane was denoted $H$}\giovanni{I made the notation uniform}
The $ \cO(0,1)$ factor imposes 
that $U_2$ be isotropic with respect to some skew-symmetric form $\omega$. Being a generic 
skew-symmetric form in five variables, $\omega$ has a one dimensional kernel $K$ that 
is not contained in $V_4$. 
This implies that when $U_1$ is contained in $V_4$, its orthogonal $U_1^\perp$ in $V_5$ with respect to $\omega$ has constant rank four. Then $U_1\wedge U_1^\perp$ is a three-dimensional subspace of 
$\wedge^2V_5$. Finally, the  $\cO(0,2)$-factor defines a quadratic form on the latter space $\wedge^2V_5$,
and restricting it to $U_1\wedge U_1^\perp$ we get our conic bundle structure on $X$. 

Note that $U_1\wedge U_1^\perp\simeq U_1\otimes (U_1^\perp/U_1)$. In terms of vector bundles, 
$\omega$ defines a morphism from $\cO_{\PP(V_5)}(-1)$ to $\cQ_{\PP(V_5)}^\vee$, which is injective 
over $\PP(V_4)$. Let $\cN^\vee$ denote the quotient of this morphism restricted to $\PP(V_4)$.
Then the conic bundle structure is defined by 
$$q: \cO_{\PP(V_4)}\longrightarrow \mathrm{Sym}^2\cE^\vee,$$

\medskip \noindent
where the bundle  $\cE=\cN\otimes\cO_{\PP(V_4)}(-1)$. Note that the restriction of $\cQ_{\PP(V_5)}$
to $\PP(V_4)$ is the sum of $\cQ_{\PP(V_4)}$ with a trivial factor, so that $\cE^\vee$ fits into an exact sequence 
$$0\lra \cO_{\PP(V_4)}\lra \cO_{\PP(V_4)}(1)\oplus \cQ_{\PP(V_4)}^\vee(1)\lra\cE^\vee\lra 0.$$

\subsubsection{The  conic bundle structure on K3-31}
Recall that Fano fourfolds $X'$ of type K3-31 are obtained as zero loci of general sections of 
\begin{equation}\tag{K3-31}
\cO(0,2) \oplus \bigwedge^2 \cQ_{\PP^3}(0,1) \oplus \cO(1,1) \xrightarrow{} \PP(V_4) \times \PP(V_7).
\end{equation}
The last two factors cut out a rank-three subbundle $\cE'$ of the trivial bundle $V_7\otimes 
\cO_{\PP(V_4)}$ over $\PP(V_4)$. More precisely, the section of $ \bigwedge^2 \cQ_{\PP^3}(0,1)$
is defined by a generic element of $\wedge^2V_4\otimes V_7^\vee$ that allows us to identify $V_7$ with 
$\wedge^2V_4\oplus\CC K$. The section vanishes at $(\ell, m)\in \PP V_4 \times \PP V_7$ when 
$m\subset \ell\wedge V_4\oplus\CC K$, which defines a rank four bundle $\cP\simeq \cQ_{\PP(V_4)}(-1)\oplus \cO_{\PP(V_4)}$  over  $\PP(V_4)$.
Then the $ \cO(1,1)$ induces an injective morphism from $\cO_{\PP(V_4)}(-1)$ to $\cP^\vee$,
whose cokernel is $\cE'^\vee$. 

 We conclude that the conic bundle structure is defined by 
$$q_1: \cO_{\PP(V_4)}\longrightarrow \mathrm{Sym}^2\cE'^\vee,$$
where the bundle  $\cE'^\vee$ fits into an exact sequence 
\begin{equation}\label{eq: exact sequence for E' in K3-31}
0\lra \cO_{\PP(V_4)}(-1)\lra \cO_{\PP(V_4)}\oplus \cQ_{\PP(V_4)}^\vee(1)\lra\cE'^\vee\lra 0.
\end{equation}

Note furthermore that 
$h^0(\cO(0,2))=28$ whereas $h^0(\mathrm{Sym}^2\cE'^\vee)=h^0(\mathrm{Sym}^2 \cP^{\vee})=27$. Since we know that the image of the projection of $\cP^{\vee})$ onto $\PP(V_7)$ is a quadric. It follows that every conic bundle on $\PP(\cE'^{\vee})$  given by a section of $\mathrm{Sym}^2\cE'^{\vee}$ corresponds to a unique complete intersection of two quadrics in $\PP(V_7)$. It remains to check that all $\cE'$ fitting in an exact sequence \cref{eq: exact sequence for E' in K3-31} arises as the zero locus of a section of  $\cO(1,1)$ restricted to $\PP(\cP^{\vee})\subset \PP(V_4)\times \PP(V_7)$. Indeed, it is enough to observe that 
$h^0(\cO(1,1))=28$ whereas $h^0(\cP^{\vee}(1))=24$ and $\PP(\cP^{\vee})\subset \PP(V_4)\times \PP(V_7)$ is defined by the vanishing of a four-dimensional space of sections corresponding to the condition $l\wedge \bar{m}=0\in \bigwedge^3 V_4$, where $\bar{m}$ is the component of $m$ belonging to $\wedge^2 V_4$.

\subsubsection{How to connect them through hyperbolic reduction} Comparing the exact sequences that 
define $\cE^\vee$ and $\cE'^\vee$, we are led to define a quadric bundle structure $Y \to \PP^3$ on the rank five bundle 
$$\cG:= \cO_{\PP(V_4)}\oplus  \cO_{\PP(V_4)}(-1)\oplus \cQ_{\PP(V_4)}(-1)$$
on $\PP(V_4)$. There is a unique way to embed in $\cG$ the trivial line bundle  $\cN=\cO_{\PP(V_4)}$.
When it is isotropic, by hyperbolic reduction  we get a conic bundle $X$ of type GM-20. 
This happens in codimension one. 

On the other hand, there are many ways to embed in $\cG$ the line bundle  $\cN'=\cO_{\PP(V_4)}(-1)$.
Such an embedding is given by $\ell\mapsto \alpha(\ell)\oplus \beta\ell \oplus \gamma\wedge\ell$
where $\alpha$ is a linear form, $\beta$ a scalar and $\gamma$ a vector (recall that  $\cQ_{\PP(V_4)}(-1)
\simeq V_4\wedge \cO_{\PP(V_4)}(-1)\subset \wedge^2V_4\otimes \cO_{\PP(V_4)}$). This gives an eight-dimensional family of embeddings, and it is easy to see that for each such embedding, there is a codimension ten family of quadric bundle structures  for which it is an isotropic line bundle.
We conclude that quadric structures on the bundles $\cG$ admitting an isotropic $\cN'=\cO_{\PP(V_4)}(-1)$ live in codimension two. 
By hyperbolic reduction through $\cN'$, we then get a conic bundle $X'$ of type K3-31.

For  a quadric threefold bundle with both isotropic subbundles $\cN=\cO_{\PP(V_4)}$ and 
$\cN'=\cO_{\PP(V_4)}(-1)$,  \cref{section} implies that  the conic bundle $X$ has a section over the hypersurface $D$ with equation $q(\cN,\cN')=0$,
which is a plane $\PP (V_3)$. Moreover this plane is embedded linearly in the Grassmannian, hence by 
sending $U_1\subset V_3$ to a plane $U_1\wedge L_0$ for some plane $L_0\not\subset V_3$. We get what is called 
a $\sigma$-plane \cite[7.1]{dim}. Gushel--Mukai fourfolds containing a $\sigma$-plane are mapped by the period map to 
the divisor $\mathcal{D}''_{10}$ in the period domain. The corresponding \HK{} fourfold is a square of a K3 surface of genus $6$ (see \cite[Remark 5.29]{dk-linear}).

On the other hand, on the side of $X'$, the fact that the trivial factor 
$\cN=\cO_{\PP(V_4)}$ is isotropic means
 by construction that the quadric $Q_1$ in $\PP(V_7)$ defined by the section of the line bundle 
$\cO(0,2)$, contains the point $[K]$. Since the other quadric $Q_2$ is the cone, with vertex $[K]$
over $G(2,V_4)\subset \PP(\wedge^2V_4)$, we deduce that $[K]$ is a singular point of $Q_1\cap Q_2$.
Moreover $X'$ contains a plane that is contracted to $[K]$ by the projection to  $\PP(V_7)$. 

\begin{prop}\label{prop:GM20-K331-hypred}
If $X'$ is a general fourfold of type K3-31 that contains a plane contracted by the projection to $\PP(V_7)$, then there exists $X$ of type GM-20 such that $X$ and $X'$ are obtained by hyperbolic reduction from the same quadric surface 
bundle $Y\ra \PP^3$. In that case the Gushel--Mukai fourfold associated to $X$ contains a $\sigma$-plane.

Conversely if $X$ is a general GM-20 fourfold associated to a Gushel--Mukai fourfold containing a $\sigma$-plane, then there exists $X'$ of type K3-31 such that $X$ and $X'$ are obtained by hyperbolic reduction from the same quadric surface 
bundle $Y \to \PP^3$. In that case $X'$ contains a plane contracted by the projection to $\PP(V_7)$.
\end{prop} 
\begin{proof}
    Consider $\mathcal S_1, \mathcal S_2\subset H^0(\mathrm{Sym}^2\mathcal G^{\vee} )$ the loci of sections inducing quadric bundles which admit an isotropic section of type $\mathcal O\to \mathcal G $ or an isotropic section of type $\mathcal O(-1)\to \mathcal G $ respectively. By the discussion above $\mathcal S_1$ is a variety of codimension $1$ in $H^0(\mathrm{Sym}^2\mathcal G^{\vee} )$ and $\mathcal S_2$ is a variety of codimension $2$ in $H^0(\mathrm{Sym}^2\mathcal G^{\vee} )$. Their intersection $\mathcal S_3=\mathcal S_1\cap \mathcal S_2$ is a variety of codimension $3$. 
    
    Any quadric bundle corresponding to a section in  $\mathcal S_1$ admits a quadric reduction with respect to the isotropic section $\mathcal O\to \mathcal G $. This leads to a fourfold of type GM-20 and every fourfold of type GM-20 appears in this way. Moreover these quadric reductions applied to quadric bundles associated to sections in $S_3$ lead to fourfolds in GM-20 which correspond to Gushel--Mukai fourfolds containing a plane. The latter locus is an irreducible variety of codimension $2$ in the moduli space of Gushel--Mukai fourfolds. Since $\mathcal S_3$ is also codimension $2$ in $\mathcal S_1$ it follows that every fourfold of type GM-20 corresponding to a Gushel--Mukai fourfold containing a plane appears as a reduction of a sixfold given by a section in $\mathcal S_3$. 
    
    Similarly, every fourfold of type K3-31 appears as a reduction by an isotropic section $\mathcal O(-1)\to \mathcal G $ of a sixfold given by a section in $\mathcal S_2$. Moreover, general sections in  $\mathcal S_3$ lead via their reduction to fourfolds of type K3-31 containing a plane contracted by the projection to $\mathbb P^7$. Since   $\mathcal S_3$ has codimension $1$ in $\mathcal S_2$ and since the locus of fourfolds of type K3-31 containing a plane contracted by the projection to $\mathbb P^7$ is irreducible of codimension at least $1$ in the moduli space of all fourfolds of type K3-31, it follows that every such fourfold appears as a reduction of a sixfold given by a section in $\mathcal S_3$. This concludes the proof in both directions.
\end{proof}

\begin{corollary}\label{coro:HS-forGM20-K331}
If $X$ and $X'$ are as in \cref{prop:GM20-K331-hypred}, then $X$ and $X'$ are $K(\PP^3)$-birationally equivalent, and $F \simeq F'$.
Moreover, we have an isometry $H^4_{pr}(X) \simeq H^4_{pr}(X')$.
\end{corollary}

\begin{proof}
    The first part is a consequence of \cite[Prop.\ 1.1]{kuznethyperbolic}. The second is a consequence of \cref{prop:cbisometry}, via $H^2(F)^-$ and the isomorphism $F \simeq F'$ compatible with the involution.
\end{proof}

Note that although \cref{prop:GM20-K331-hypred} holds only for a subfamily of fourfolds of type K3-31 we can still prove that even nodal surfaces of K3 type associated to fourfolds of type K3-31 are also associated to some special fourfolds of type GM-20.
\begin{prop}
    Let $X'$ be a general fourfold of type K3-31 and let $\Delta_{X'}\subset \mathbb P^3$ be the discriminant locus of the associated conic fibration. Then there exists $Y$ a fourfold of type GM-20 such that the discriminant locus of its conic fibration is $\Delta'$.   
\end{prop}
\begin{proof} Let us fix a general K3 surface $S$ of genus 6 in $\mathbb P^6$ and let $M$ be the EPW sextic associated to $S$. More precisely consider the rational map $\psi\colon S^{[2]}\to \PP(H^0(I_S(2),\PP^6)^{\vee})$ associating to a scheme $Z$ of length 2 in $ S$ the hyperplane of quadrics containing $S$ and the line spanned by $Z$. The image of that map is an EPW sextic that we call $M$. Let now $X'$ be the intersection $Q_1\cap Q_2$ of two quadrics $Q_1$, $Q_2$ containing $S\subset \PP^6$. Consider now the codimension two linear section $L\cap M$ of $M$ where $L$ is the linear space of hyperplanes in $H^0(I_S(2),\PP^6)$ containing the space $H^0(I_{Q_1\cap Q_2}(2),\PP^6)$. We claim that $L\cap M\simeq \Delta'$. Indeed, consider the rational map $\phi\colon \PP^6\setminus S \to \PP^5$ defined by quadrics containing the K3 surface $S$. It maps $Q_1\cap Q_2$ to $L$ and its fibers are conics meeting the K3 surface $S$ in four points and form the rational conic fibration corresponding to K3-31. In that case, general points of $\Delta'$ are those points in $L$ over which the fiber is a reducible conic consisting of two lines $l_1 \cup l_2$ each meeting the K3 surface in 2 points (there is no line meeting $S$ in more than two points) on $Q_1\cap Q_2$. But for any line $l$ meeting $S$ in two points the image of $l$ through $\phi$ is the same as the image via $\psi$ of the pair of points $l\cap S$. It follows that $\Delta'= L\cap M$.
It follows that the discriminant loci of conic fibrations on general fourfolds of type K3-31 are exactly the general codimension 2 sections of special (singular) double EPW sextics  which are birational to $S^{[2]}$ for $S$ a K3 surface $S$ of genus 6.

On the other hand, we have seen in \cref{codim2} that the discriminant locus of a fourfold of type GM-20 is just a general hyperplane section of the singular hyperplane section of its corresponding EPW sextic defining the Hilbert scheme of lines on the GM fourfold. It remains to observe that for a general codimension 2 section of an EPW sextic $X$ admits six singular hyperplane sections of $X$ containing it; these are given by the intersection of the orthogonal line to the codimension 2 section with the dual EPW sextic. In particular, also for a codimension 2 linear section of a special EPW sextic corresponding to $S^{[2]}$ for $S$ a K3 surface $S$ of genus 6 we will have six hyperplane sections  corresponing to points on the dual EPW sextic each corresponding to a special fourfold of type GM-20 where the Gushel Mukai fourfold corresponds to a point in the divisor $\mathcal D''_{10}$ in the period domain.  The codimension 2 section will then appear as a discriminant locus of each of the six fourfolds of type GM-20.

It follows that an even nodal surface of K3 type arising as a discriminant locus of the conic bundle structure on a general fourfold of type K3-31 arises also as a discriminant locus of the conic bundle structures of six fourfolds of type GM-20.   
\end{proof}

We thus have concluded the proof of \hyperlink{main_theorem_b}{Main Theorem B} for this case.
\begin{remark}
    Note that all fourfolds of type K3-31 as well as all fourfolds of type GM-20 corresponding to Gushel--Mukai fourfolds containing a plane are known to be rational varieties i.e.\ in particular they are all birational to each other. However, the problem of birationality over $K(\PP^3)$ of a general fourfold of type K3-31 with one of the six fourfold of type GM-20 inducing the same discriminant locus still remains open and has probably a negative answer. It is in fact not hard to check that the cokernel sheaf of a conic bundle associated to a general fourfold of type K3-31 is different from cokernel sheaves of the conic bundles associated to the six fourfolds of type GM-20 with the same discriminant locus, this by results of Kuznetsov \cite{kuznetconicbundles} is enough to conclude that these conic bundles are not connect by hyperbolic reduction with each other (see \cref{openProblems} for a more detailed discussion).
\end{remark}
%\begin{prop}
 %   If $X'$ is a general fourfold of type K3-31 then there exists another fourfold $X"$ of type K3-31 containing a plane contracted by the projection to  $\PP(V_7)$ such that $X"$ is obtained from $X'$ by elementary transformations. Then in particular $X'$ and $X"$ are birational over $K(\PP^3)$, the double covers $F'$ and $F"$ of their  discriminant loci are isomorphic and we have an isometry $H^4_{pr}(X')=H^4_{pr}(X")$.
%\end{prop}

\section{20-Nodal quartics on quadrics and Gushel--Mukai fourfolds}
\label{sect20nodal}
   
\subsection{Blow-up of a GM along a quadric and nodal quartics} Let $X_{10}$ be a general Gushel--Mukai fourfold, and $\QQ^3$ the quadric threefold. We first recall how to construct a conic bundle structure $X_{10} \dashrightarrow \QQ^3$. Inside $\Gr(2,4) \times \Gr(2,5)$, let $X$ be the zero locus of a general section of $\cO(0,2) \oplus \cQ_{\Gr(2,4)} \boxtimes \cU^\vee_{\Gr(2,5)} \oplus \cO(1,0)$.
As described in \cite{bfmt}, we have that:

%\subsubsection*{Invariants}  $h^0(-K)=30, \ (-K)^4=114, \  h^{1,1}=2, \ h^{3,1}=1, \ h^{2,2}=24$, $-\chi(T)=24$.

\begin{itemize}[leftmargin=*]
    \item the first projection $\pi_1 \colon X \to \Gr(2,4)$ is a conic bundle on $\QQ^3$ with discriminant divisor $\Delta$ a quartic surface with $20$ nodes.
    \item The second projection $\pi_2 \colon X \to \Gr(2,5)$ is a blow-up $Bl_{\PP^1 \times \PP^1}X_{10}$  of a Gushel--Mukai fourfold along a $\sigma$-quadric surface. (Remember from \cite[section 3]{dim} that 
   a general Gushel--Mukai fourfold contains a unique $\sigma$-quadric surface.)
\end{itemize}

\medskip The description of $X$ is as follows. Start with two vector spaces $V_1$ and $V_4$, and let $V_5=V_1\oplus V_4$.  Inside $\Gr(2,V_4)\times \Gr(2,V_5)$ lies the Grassmann
bundle $\Gr(2,\cO  \oplus \cU)$ over $\Gr(2,V_4)$, where $\cU$ denotes the tautological rank two bundle over $\Gr(2,V_4)$ and $\cO$ is the trivial line bundle $V_1 \otimes \cO$. Then $X$ is the zero locus of sections of the line bundles $\cO(1,0)$
and $\cO(0,2)$. The former reduces us to $\QQ^3\subset \Gr(2,V_4)$, and the latter to 
conics inside the fibres $\Gr(2,\cO  \oplus \cU)=\PP(\cO \oplus \cU^\vee)$. Hence the conic bundle structure. 

The discriminant double cover $F \to \Delta$ we want to describe parametrizes lines in $X$, contracted by $\pi_1$. 
So we should be looking at the fibre bundle $\cL=\Gr(1,V_1\oplus L)$ over $\Gr(2,V_4)$,
that parametrizes lines in $\Gr(2,V_5)$. Indeed for any $\ell\subset V_1\oplus L$, the 
corresponding line is the projectivisation of $\ell\wedge (V_1\oplus L)\simeq \ell\otimes 
(V_1\oplus L)/\ell$. So the surface $F$ is defined inside the $6$-dimensional fibre bundle $\cL$
as the zero locus of a section of the rank four vector bundle  $\cO(1,0)\oplus 
S^2(\ell\wedge (V_1\oplus L))^\vee$. This  description of $F$ as the zero loci of a section of a vector bundle allows us to compute the numerical invariants of this surface.

\begin{prop} 
The surface $F$ is a minimal surface of general type. 
The numerical invariants of $F$ are 
$$\chi(\cO_F)=7, \quad e(F)=68, \quad K_F^2=16, \quad \chi(-K_F)=23.$$
Moreover, $\chi(T_F)=-38$.
\end{prop} 

\begin{comment}
Here is the code used to determine those invariants:
restart
load "hodgeForProductsOfFlag.m2"

Gr24=flagBundle({2,2})
bundlesGr24=bundles Gr24
UGr24=first bundlesGr24
Lcal=flagBundle({1,2},OO_Gr24 + UGr24,VariableNames=>K)
--L=first bundles Lcal
--L=OO_(Lcal)(1)
(Lcal.QuotientBundles) / (c -> rank c)
F=((Lcal.QuotientBundles)_1)
O1rel=first bundles Lcal
--O1rel=OO_Lcal(1)
--
O1dalbasso=(Lcal / Gr24)^*(det dual first bundlesGr24)
chern(1,O1dalbasso)
chern(1,O1rel)
chern(1,F)

E=O1dalbasso + dual symmetricPower(2,F*O1rel)
S=sectionZeroLocus(E)
euler(S)
vol:=integral product apply(dim S,i->(chern(1,tangentBundle(S))))
h0:=integral(todd(tangentBundle(S))*ch(det(tangentBundle(S))))
chi(OO_S)
chi(tangentBundle(S))
\end{comment}

\begin{proof}
As a double cover of a quartic surface in $\QQ^3$ (i.e.\ a complete intersection of type $(2,4)$ in $\PP^4$) ramified in 20 nodes, the canonical divisor of $F$ is ample from the Nakai--Moishezon criterion.
%\giovanni{Citation missing here}The description of $F$ as a zero locus of a section of a vector bundle allows  to compute its numerical invariants with the aid of \cite{M2}.
\end{proof}

\begin{corollary}\label{cor:F-Simply-conn-GM21}
The surface $F$ is algebraically simply connected.
\end{corollary}

\begin{proof}
This follows again from the fact that $3K^2 < 8(\chi(\cO_F)-2)$, see \cite[Cor.\ 4]{Xi87}.
\end{proof}

\begin{prop}
Up to multiplication by -2, there is a Hodge isometry between $H^2(F)^{-}$ and $H^4_{pr}(X,\ZZ)$. In particular, $H^2(F)^-$ has a K3-type Hodge structure with $h^{1,1}=22$.
\end{prop}

\begin{proof}
The isometry is given by \cref{prop:cbisometry}. Since $X \to X_{10}$ is a blow-up of a quadric, we have that $h^{2,2}(X)=24$. On the other hand $H^4_{pr}(X,\ZZ)$ is obtained as the orthogonal to two algebraic classes, so that $h^{2,2}_{pr}(X)=22$.
\end{proof}

We have then proved \hyperlink{main_theorem_b}{Main Theorem B} (5) and (6) in this case.

%\marcello{Maybe the next paragraphs need some cleaning. For example, do we really want to keep the first one?}
%\laurent{I would keep it, as a double-check}
%\enrico{I agree, it is also a very neat computation}
Note that using the same approach as in \cite{hu21} we can also calculate some invariants of $F$. In fact, if we take a minimal resolution of a quartic surface $\tilde{\Delta}\subset \QQ^3$ (which is diffeomorphic to a smooth octic complete intersection in $\PP^4$), it has $e(\tilde{\Delta})=64$ and $K^2_{\tilde{\Delta}}=8$. This implies that the nodal $\Delta$ has $e(\Delta)=44$. Hence, the double cover surface $F$ will have $e(F)=2 \cdot 44-20=68$ and $K^2_F= 2 K^2_\Delta=2K^2_{\tilde{\Delta}}=16$. It follows that $12\chi(\cO_F)= K^2_F+68$, i.e.\ $\chi(\cO_F)=7$.

Note that the line of planes $P$ such that $\ell\subset P\subset V_1\oplus L$ contains 
$P=V_1\oplus \ell$ (if $\ell$ is not $V_1$), and it follows that $F$ can also be described 
as the surface of lines in $X$ that touch the $\sigma$-quadric $Q_0\subset X$. Indeed the latter consists in the planes in $X$ passing through the kernel of the skew-symmetric form that 
cuts-out the hyperplane section of $\Gr(2,5)$ containing $X$, and this kernel is precisely $V_1$. 
So with the notations of \cite[p. 22 of arXiv version]{dk-linear}, $F=p(q^{-1}(Q_0))$.

Now, according to the proofs of \cite[Lemma 5.10 and Corollary 5.11]{dk-linear}, 
$F$ should be the preimage of the intersection $S$ of the EPW-sextic $\tilde{Y}_A$
with a (singular) quadric  $C_{v_0}Q_0\subset \PP(V_5)$. But astonishingly, the 
double cover should split into the union of two surfaces $F=R\cup\tau_A(R)$, where 
$\tau_A$ is the antisymplectic involution of  $\tilde{Y}_A$ whose quotient is $Y_A$. 

%\begin{prop}(???)
%The singular surface $\Sigma=S$ is a section of the EPW-sextic $Y_A$ by a hyperplane and 
%a quadric. 
%\end{prop}
\smallskip Recall that, as proved in \cite[Theorem 3.6]{dk-class}, 
there is 
a bijection between the set of smooth ordinary Gushel--Mukai fourfolds $X$, and the set of 
Lagrangian data $(A,V_6,V_5)$. This was suggested in \cite[section 4.5]{im} with the following geometric interpretation: the Hilbert scheme $F_c(X)$ of conics in $X$ is essentially a $\PP^1$-bundle over 
the double EPW-sextic $\tilde{Y}_A$, and $V_5$ corresponds to the Pl\"ucker point on 
$Y_A\subset \PP (V_6^\vee)$. 

On the other hand, the following incidence variety was defined in \cite{kkm}:
$$K_A=\{([U],[\alpha]) \in Z_A\times\PP(\wedge^3V_6), \quad [\alpha] \in \PP(T_U)\cap \PP(A)\cap \Omega\},$$
where $\Omega\subset \PP(\wedge^3V_6)$ is the variety of partially decomposable three-forms and $Z_A$ the EPW cube associated to $A$. 
In particular $\Omega$ contains $\Gr(3,6)$, and the complement has two natural fibrations
over $\PP(V_6)$ and $\PP(V_6^\vee)$. This induces two projections whose images are the 
$EPW$ sextic $Y_A$ and its dual $Y_A^\vee$. It was observed in \cite[Lemma 3.3]{kkm} that the 
fibres of these projections are quartic sections of $\QQ^3$ with an even set of $20$ nodes. 

\begin{prop} The singular surface $\Delta$ is the fibre over the Pl\"ucker point of $X$ 
of the projection of $K_A$ to $Y_A$.
\end{prop}

%\laurent{$\Sigma$ vs $\Delta$}
Note as a consequence that through the projection of $K_A$ to $Z_A$, 
the surface  $\Delta$ embeds into the EPW-cube $Z_A$, and its double cover $F$ should 
thus embed into the double EPW-cube $\tilde{Z}_A$ (see \cite[Remark 3.4]{kkm}). 
%\laurent{Does it?} \giovanni{as the double cover is unique, we should just say that it does embed}
In fact the image of $\Delta$ appears as the intersection of the EPW cube $Z_A$ with a special 3-dimensional quadric contained in the quartic degeneracy locus $D^1_{A,T}$.

\medskip\noindent {\it Remark.} $40$-nodal quartics in $\QQ^3$ 
also appear in \cite{cmr} as images
of Schoen surfaces by their canonical systems.

\subsection{Another conic bundle with quartic discriminant} Let $X'$ in 
$\Gr(2,4)_1 \times \Gr(2,4)_2$ be the zero locus $X$ of a general section of $\cO(0,1) \oplus \cU^{\vee}_{\Gr(2,4)_2}(1,0) \oplus \cO(1,1)$. As described in \cite{bfmt}, case K3-35:
%\subsubsection*{Invariants}  $h^0(-K)=25, \ (-K)^4=90, \  h^{1,1}=2, \ h^{3,1}=1, \ h^{2,2}=24$, $-\chi(T)=23$.
\begin{itemize}[leftmargin=*]
    \item the first projection $\pi_1 \colon X' \to \Gr(2,4)_1$ is birational, with base locus a genus $6$ K3 surface with two double points,
    \item $\pi_2 \colon X' \to \Gr(2,4)_2$ is a conic fibration over $\QQ^3$, with discriminant a quartic hypersurface with $20$ nodes.
\end{itemize}

The above description allows us to calculate the invariants of $F'$, that can be defined as the zero locus of a general section of some vector bundle, and show that they indeed coincide with the ones of $F$. We omit these calculations, since we will prove that $F'$ is isomorphic to some $F$ as in the the GM-20 case. Indeed, the upshot is that any such $X' \to \PP^3$ is related to a special case of an $X \to \PP^3$ of type GM-20. Here there is no evidence from the moduli of nodal sextic, that the discriminant double cover $F' \to \Delta'$ must be isomorphic to some $F \to \Delta$. We will obtain this isomorphism by hyperbolic reduction.

%\marcello{I commented out the constructions that allow us to calculate the invariants of $F'$, since the result is redundant : we will prove via hyperbolic reduction that $F'$ is isomorphic to some $F$. But if you want to add them back to the paper, just do it!}

\begin{comment}
In this case the surface $F$ can be described as follows. Consider the product 
$Fl(1,3,A)\times \Gr(2,B)$ for two four-dimensional spaces $A$ and $B$. Let $A_1\subset A_3\subset A$
denote the tautological bundles on $Fl(1,3,A)$, and $B_2\subset B$ the tautological bundle 
on $\Gr(2,B)$. Then $F$ can be obtained as the zero locus of a section of the rank seven
vector bundle $$\wedge^2B_2^\vee\oplus (A_1\wedge A_3)^\vee\otimes B_2^\vee
\oplus (A_1\wedge A_3)^\vee\otimes\wedge^2B_2^\vee.$$

\begin{prop}
The numerical invariants of the surface $F$ are 
$$p_g=6, \quad q=0, \quad h^{1,1}=54, \quad K_F^2=16, \quad \chi(-K_F)=23.$$
Moreover, $\chi(T_F)=-38$.
\end{prop}
\end{comment}

\begin{comment}
Here is the code used to determine those invariants:
restart
load "hodgeForProductsOfFlag.m2"
Fl=flagBundle({1,2,1})
bunFl=bundles Fl
X=flagBundle({2,2},OO_Fl^4)
bunFl=bunFl / (b->(X/Fl)^*(b))
bunGr=bundles X
Z=exteriorPower(2,dual bunGr_0) + 
  (dual (bunFl_0 * bunFl_1)) * (dual bunGr_0) +
  (dual (bunFl_0 * bunFl_1)) * (exteriorPower(2,dual bunGr_0))
S=sectionZeroLocus(Z)
euler(S)
vol:=integral product apply(dim S,i->(chern(1,tangentBundle(S))))
h0:=integral(todd(tangentBundle(S))*ch(det(tangentBundle(S))))
chi(OO_S)
chi(tangentBundle(S))
\end{comment}

\medskip

This is the type of surfaces
that appears also in \cite{kkm} in relation with EPW sextics.
%\laurent{seems?}
The fact that one can associate a special EPW sextic to a K3 surface of genus $6$ is 
discussed in \cite[\textsection 4.3]{og-inv}. It was actually observed by Mukai at the very beginning of the story.

Let us try to count parameters. The fourfold $X$ is defined by three tensor $\alpha\in \wedge^2B^\vee$, $\beta\in \wedge^2A^\vee\otimes\wedge^2B^\vee$, $\gamma\in \wedge^2A^\vee\otimes B^\vee$. Multiplying these tensors by nonzero scalars does not
make any difference, nor changing $\beta$ by some tensor constructed from $\alpha$ and $\gamma$, that is by some $\delta\beta=\Omega\otimes \alpha+\gamma\wedge\phi$ for some 
$\Omega\in \wedge^2A^\vee$ and $\phi\in B^\vee$. For $(\alpha,\gamma)$ general the induced map 
from $\wedge^2A^\vee\oplus B^\vee$ to $\wedge^2A^\vee\otimes\wedge^2B^\vee$ is injective, 
and its cokernel  therefore has dimension $36-6-4=26$. Modulo the action of $SL(A)$
and $SL(B)$ we would thus expect $(6-1)+(24-1)+(26-1)-(15+15)=23$ parameters. 
This is a little strange if we imagine that the surface from \cite{kkm} should
be a deformation of ours (corresponding to a deformation from a special EPW sextic
to a general one). 

On the other hand, the linear system of quartic sections of a quadric has dimension 
$70-15-1=54$, which reduces to $44$ modulo the action of $SO_5$. Imposing $20$ ODP 
should reduce the number of parameters to $24$.

% \bigskip

\subsection{Comparing the conic bundles}

\subsubsection{The  conic bundle structure on GM-21}
Recall that GM-21 is the family of Fano fourfolds $X$ defined as zero loci of general sections of 
\begin{equation}\tag{GM-21}
\cO(0,2) \oplus (\cQ_{2,V_4} \boxtimes \cU^\vee_{2,V_5}) \oplus \cO(1,0) \xrightarrow{} \Gr(2,V_4) \times \Gr(2,V_5).\end{equation}
The section of $\cQ_{2,V_4} \boxtimes \cU^\vee_{2,V_5}$ is defined by a general morphism 
from $V_5$ to $V_4$,
that we can identify with the projection to $V_4$ with respect to a decomposition  $V_5=V_4\oplus \CC K$.
The zero locus of this section is then the set of pairs $(U_2,U'_2)$ such that $U'_2\subset U_2
\oplus\CC K$. When $U_2$ is fixed, we get inside $\Gr(2,V_5)$ a projective plane  $\Gr(2,U_2
\oplus\CC K)\simeq\PP^2$, in which the factor  $\cO(0,2)$ cuts-out a conic bundle. 
The factor  $\cO(1,0)$ imposes that $U_2$ be isotropic with respect to some 
symplectic form on $V_4$, or equivalently that $[U_2]$ belongs to some smooth hyperplane section 
of $\Gr(2,V_4)$, which is a copy of $\QQ^3$. So the conic bundle structure is defined by a map
$q: \cO_{\QQ^3} \longrightarrow \mathrm{Sym}^2\cE^\vee,$
with 
$$\cE=\wedge^2(\cU\oplus\cO_{\QQ^3})\simeq \cO_{\QQ^3}(-1)\oplus\cU.$$
Indeed, in this case the system $\cO_{\mathrm{P}(\cE^{\vee})}(1)$ gives the second projection to a hyperplane section $R$ of $\Gr(2,5)$ embedded by the Pl\"ucker embedding in $\PP^8$.
It remains to observe that any conic bundle given by a map
$q: \cO_{\QQ^3} \longrightarrow \mathrm{Sym}^2\cE^\vee,$ corresponds to a section on  $\cO_{\mathrm{P}(\cE^{\vee})}(2)$, which itself corresponds to a quadric section of the Grassmannian. To see the latter we make a dimension count: 
$$H^0(\mathrm{Sym}^2\cE^\vee)=40=H^0(\cO_{R}(2)).$$
Summing up the manifolds of type GM-21 are exactly the conic bundles associated to  maps $\cO_{\QQ^3} \longrightarrow \mathrm{Sym}^2\cE^\vee$.

\subsubsection{The  conic bundle structure on K3-35}
Let us compare with  the description of K3-35 as the family of Fano fourfolds $X'$ obtained as zero loci of general sections of 
\begin{equation}\tag{K3-35}
\cO(1,1) \oplus  \cU^\vee(0,1) \oplus \cO(1,0) \xrightarrow{} \Gr(2,U_4) \times \Gr(2,V_4),
\end{equation}
where $\cU$ is the tautological rank two bundle pulled-back from  $\Gr(2,U_4) $. 
Here again the  $\cO(1,0)$ factor cuts out a quadric $\QQ^3$ in $\Gr(2,U_4)$ which is the base of the conic bundle. % We can think then of the two families of fivefolds $Y \to \Gr(2,V_4)$ which are conic bundles over the Grassmannian. As explained above, in the first case (GM21) the conic bundle lives in the projectivisation of the rank three bundle $\cE = \cO \oplus \cU^\vee = \cO \oplus \cU(1)$ on $\Gr(2,V_4)$.
Over $U_2\subset U_4$, the factors  $\cO(1,1)$ and $\cU^\vee(0,1)$ cut-out a three-dimensional subspace 
of $\wedge^2V_4$, on which we can restrict the Pl\"ucker quadric. 
So the conic bundle structure is defined by a map $\cO_{\QQ^3}\longrightarrow \mathrm{Sym}^2\cF^\vee,$ where the rank-three vector bundle $\cF$ is defined by an exact sequence 
\begin{equation}%\tag{Rk3 bundle K3-35}
0 \longrightarrow \cU \oplus \cO_{\QQ^3}(-1) \longrightarrow \wedge^2 V_4^\vee \otimes \cO_{\QQ^3} \longrightarrow \cF^\vee \longrightarrow 0.
\end{equation}
Let us give another presentation of the bundle $\cF$.
Write $\wedge^2 V_4 = V_6$, and consider the diagram
$$
\xymatrix{
 & \cU \ar[d] & \cQ^\vee \ar[d] & & \\
 0 \ar[r] & \cU \oplus \cO(-1) \ar[d]\ar[r] & V_6 \otimes \cO \ar[d]\ar[r] & \cF^\vee \ar[r]& 0 \\
 0 \ar[r] & \cO(-1) \ar[r] & \cU^\vee \oplus \cO^{\oplus 2} \ar[r] & \cE'^\vee \ar[r] & 0 } 
$$
where the middle column is the tautological sequence with an extra trivial rank two bundle in the cokernel (and both columns are exact). Since $\cU \simeq \cQ^\vee$ in restriction to $\QQ^3$,  we deduce
that $\cE' = \cF$.

\subsubsection{How to connect them through hyperbolic reduction}
Consider the rank five vector bundle 
$$\cG := \cU \oplus \cO^{\oplus 2} \oplus \cO(-1),$$
and a quadratic form $Q:\cO_{\QQ^3}\ra \mathrm{Sym}^2\cG^\vee$ on it. Let $Y \to \QQ^3$ be the associated quadric bundle. 
Restricting to $\cO^{\oplus 2}$ we get a constant quadratic form
in two dimensions, so there always exist a trivial isotropic subbundle $\cN$ of $\cG$. 
By hyperbolic reduction, we get a conic bundle $X$ of type GM-21. 

If we moreover assume that $\cN'=\cO(-1)$ is isotropic, we can perform another hyperbolic reduction.
The exact sequence \cref{eq:useful-for-hyp-red} shows that the resulting conic bundle will be of type  K3-35. Applying
\cref{section}, we get:

\begin{prop} 
If $X'$ is a general fourfold of type K3-35 then there exists a fourfold $X$ of type GM-21
such that its associated Gushel--Mukai fourfold contains a $\tau$-quadric and a quadric threefold bundle $Y\ra \QQ^3$ such that both $X$ and $X'$ are obtained by hyperbolic reduction from $Y$.  

Conversely, if $X$ is a fourfold of type GM-21 corresponding to a general Gushel--Mukai fourfold containing a $\tau$-quadric, then there exists $X'$ a fourfold of type K3-35 and $Y\ra \QQ^3$ such that both $X$ and $X'$ are obtained by hyperbolic reduction from $Y$.  
\end{prop} 
\begin{proof}
    If $X'$ is a general fourfold of type K3-35 then $X'$ has a conic bundle structure $\QQ^3$ corresponding to a map:
    $q_{X'}:\cO_{\QQ^3} \longrightarrow \mathrm{Sym}^2\cE'^\vee$ with 
    $\cE'$ defined by:
    $$0\longrightarrow \cO(-1)\longrightarrow \cU^{\vee}\oplus \cO^{\oplus 2}\longrightarrow \cE'^{\vee}.$$
    By \cite[Proposition 3.6]{BKK} we deduce that $X'$ appears as a conic reduction of a quadric threefold bundle $Y$ associated to a map 
    $q_{Y}:\cO_{\QQ^3} \longrightarrow \mathrm{Sym}^2\cG^{\vee}$ with 
    $$\cG=\cU\oplus \cO^{\oplus 2}\oplus \cO(-1).$$ This conic bundle $q_Y$ always admits an isotropic section $\cO\to \cG$, and a reduction with respect to this section gives rise to a fourfold $X$ of type GM-21 which is a conic bundle associated to a map 
    $q_{X}:\cO_{\QQ^3} \longrightarrow \mathrm{Sym}^2\cE^{\vee}$, where 
    $\cE=\cO(-1) \oplus \cU$. Moreover, since the quadric threefold bundle $Y$ admits two isotropic sections $\cO\to \cG$ and $\cO(-1)\to \cG$, by \cref{section} its reduction $X'$ admits a section over the divisor $D\subset \QQ^3$ where the two sections are orthogonal with respect to $q_Y$. This divisor $D$ is  just a hyperplane section of $\QQ^3$ and is mapped by the section to a $\tau$-quadric that is disjoint from the $\sigma$-quadric in the Gushel--Mukai fourfold. Note that in fact, we always have two isotropic sections $\cO\to \cG$,  giving rise to two such $\tau$-quadrics.
    
    Conversely, let $G$ be a Gushel--Mukai fourfold containing a $\tau$-quadric surface $T$. The surface $T$ is associated to a fourspace $V_4$ such that $T\subset \Gr(2,V_4)$. Furthermore there is another $\tau$-quadric surface $T'$ obtained as the residual component to $T$ in the intersection $G\cap \Gr(2,V_4)$. Now, as described above, the corresponding fourfold $X$ of type GM-21 is described as the blow up of $G$ in its $\sigma$-quadric $S$. The latter defines a unique one-dimensional subspace $V_1\subset V_4$ such that $S\subset \Gr(V_1,2,V_5)$. We thus have a natural decomposition $V_5= V_4\oplus V_1$ and a fourfold $X$ with the corresponding conic bundle structure. Let $q_X: \cO\longrightarrow \mathrm{Sym}^2 (\cO(1) \oplus \cU^{\vee})$ be map associated to the conic bundle $X$. Let us write it as a block matrix:
   \[\left(
\begin{array}{cc}
    \phi & a \\
    a & q
\end{array}\right),
\]
with $\phi\in H^0(\mathrm{Sym}^2\cU^{\vee}) $, $a\in H^0(\cU^{\vee}(1)) $ and $q\in H^0(\cO(2))$.

    Observe that in that description $\QQ_3$  is just the hyperplane section of $\Gr(2,V_4)$ by the hyperplane spanned by the Gushel--Mukai fourfold $G$. Moreover, under this identification the component $q$ of $q_X$ corresponding to $\cO_{\QQ^3}(2) $ defines the quadric section of $\QQ^3$ corresponding to $G\cap \Gr(2,V_4)$. By assumption, the latter decomposes into the product $q=l_1l_2$ of two sections $l_1$, $l_2$ of $\cO_{\QQ^3}(1)$. We can now consider the map $q_G:\cO\longrightarrow \mathrm{Sym}^2 \cG^{\vee}$ given by the following block matrix:
         \[\left(
\begin{array}{cccc}
    \phi & 0& 0&a \\
    0 & 0& 1& l_1\\
    0&1 & 0 & l_2\\
    a & l_1&l_2& 0
\end{array}\right),
\]
and observe that the corresponding quadric threefold bundle has two isotropic sections $\cO\to \cG$ and one isotropic section $\cO(-1)\to \cG$. Furthermore, taking the reduction with respect to one of the two sections  $\cO\to \cG$ leads to $X$, whereas taking the reduction with respect to the isotropic section $\cO(-1)\to \cG$ leads to a conic bundle structure on a fourfold $X'$ of type K3-31.
\end{proof}

Gushel--Mukai fourfolds containing a $\tau$-quadric 
are parametrised by a divisor $\mathcal{D}'_{10}$ in the period domain \cite[7.3]{dim}.

\begin{prop}
The $20$-nodal quartic section of $\QQ_3$ in the case GM-21 is a  section of a general EPW cube contained in $\Gr(3,6)$ with a special $\Gr(2,4)\subset  \Gr(3,6)$.
In the special cases the corresponding double EPW cube is associated to a Lagrangian space $A$ such that $A^{\perp}$ is contained in $\Delta$ (notation from \cite{dk-linear}). 
%is birational to a cube of a genus $6$ K3 surface.
\end{prop}
%\laurent{"corresponds" is too vague}
\begin{proof}
  The $20$-nodal surface is  constructed as follows. Let $\U'$ be the tautological bundle on the Grassmannian $\Gr(2,4)$ that restricts to the sheaf $\U$ on the hyperplane section $\QQ_3$.
Let us consider the symmetric degeneracy locus constructed from the bundle $$0\to \U\oplus \cO(-1)\to \U^{\vee}\oplus \cO(1)\to \eta \to 0$$
where $\eta$ is a symmetric sheaf supported on a $20$-nodal surface cut out by a quartic in $\QQ_3\subset \PP^4$.
%Having such a description we can deduce by Lefschetz that the examples are algebraically simply connected. 

    We claim that the surface $(2,4)$ constructed in an EPW cube can be constructed as a symmetric degeneracy locus as before.

Let $p=[v\wedge \alpha]$ be a point in $\Omega$ and $A\subset \bigwedge^3 V_6$ be a Lagrangian subspace such that $\mathbb{P}(A)\cap F_{[v]}=p$. Recall that the $(2,4)$ surface $S_{p,A}$ in the EPW cube $Y_A$ corresponding to $p$ is defined as follows.
\begin{itemize}[leftmargin=*]
\item First let $Q_3=\{[U]\in \Gr(3,6) : \mathbb{P}(T_{[U]})\ni p\}$. Note that this is indeed a three-dimensional quadric, since it is a hyperplane section of the Grassmannian $\Gr([v],3,[v\wedge \alpha\wedge \alpha])\simeq \Gr(2,4)$ parametrizing three-planes containing the line $[v]$ and contained in the hyperplane $[v\wedge \alpha\wedge \alpha]$; the hyperplane section is  defined by the condition $\wedge^2 U\wedge v\wedge \alpha=0$.
\item Then $S_{p,A}=D^2_A\cap Q_3$ can be described as the first Lagrangian degeneracy locus $D^1_{\bar{\mathcal T}, \bar A}$ of the family $\bar{\mathcal T}$ of Lagrangians  $T_{[U]}/\langle p\rangle \,\subset\, \langle p\rangle ^{\perp}/\langle p\rangle $ for $[U]\in Q_3$ with respect to $\bar A=A/\langle p\rangle \,\subset\, \langle p\rangle ^{\perp}/\langle p\rangle $.
\end{itemize}
Consider $P_{[U]}:=v\wedge U\wedge V_6$ for $[U]\in Q_3$. Clearly we have $P_{[U]}\subset   T_{[U]}$ and $\dim P_{[U]}=7$. In fact, $\mathbb{P}(P_{[U]})=\mathbb P(T_{[U]}\cap F_v)$ is the projective tangent space to the Grassmannian $\Gr([v], 3, V_6)$ in $[U]$ whereas $\mathbb P(T_{{U}})$ is the projective tangent to $\Gr(3,V_6)$ in $[U]$. Let us furthermore observe that for $[U]\in Q_3$ we have $p\in \mathbb P(P_{[U]})$. Indeed, the condition $\wedge^2 U\wedge v\wedge \alpha=0$ implies that $v\wedge \alpha =v\wedge u_1\wedge v_1+v\wedge u_2 \wedge v_2$ for some $u_1,u_2\in U$ and $v_1,v_2 \in V_6$. Hence, $p=[v\wedge \alpha] \in P_{[U]}$. Finally, observe that under our assumptions $A\cap P_U=\langle p\rangle $ since $P_U\subset F_v$, and $\langle p\rangle \subset P_U$, and $A\cap F_v=\langle p\rangle $. 

Now for $[U]\in Q_3$ let $A_{P,U}=(A\cap P_U^{\perp})/P_U \subset P_U^{\perp}/P_U$ and similarly $T_{P,U}=(T_{[U]} \cap P_U^{\perp})/P_U= T_{[U]}/P_U$. Then both $A_{P,U}$ and $T_{P,U}$ are Lagrangian subspaces of the space  $P_U \subset P_U^{\perp}$. Taking the above into account we deduce that $S_{p,A}$ is the first degeneracy locus of the relative Lagrangian degeneracy locus between the Lagrangian subbundles $\mathcal A_P$ and $\mathcal T_P$ of the symplectic  bundle $\mathcal P^{\perp}/\mathcal P$ swept by $A_{P,U}$ and $T_{P,U}$ respectively. If we now consider additionally the subbundle $\mathcal F_{v,P}$ swept by $F_{v,P,U}=F_v/P_U$ for $[U]\in Q_3$ then we see that $F_{v,P,U}\cap A_{P,U}=0$ as well as $F_{v,P,U}\cap T_{p,U}=0$. By the symplectic form we can now identify 
$\mathcal F_P$ with $\mathcal T_P^{\vee}$ and then $A_{P,U}$ is the graph of a symmetric map $Q_{A,P}:\mathcal T_P\to \mathcal T_P^{\vee}$. In this way $S_{A,p}$ is seen as a symmetric degeneracy locus of the conic bundle given by
$A_{p,U}$ on $\mathbb{P}(T_P)$. It remains to observe that $\mathcal T_P(1)\simeq \mathcal U^{\vee}\oplus \mathcal O$. Indeed $\mathcal T_P(1)=\mathcal T_{\Gr(3,6)}/\mathcal T_{\Gr(1,3,6)}|_{Q_3}=T_{\Gr(3,6)}/T_{\Gr(1,3,6)}|_{Q_3}=((\mathcal O\oplus\mathcal U)\otimes \mathcal Q)/(\mathcal U\otimes \mathcal Q)=\mathcal Q=\mathcal U^{\vee}\oplus \mathcal O$.

The Lagrangian space is described in \cite[Remark 5.29]{dk-linear}.
%The fact that in the special case the corresponding EPW sextic is birational to a square of a genus $6$ K3 surface follows from \cite[Proposition 7.4]{dim}.
%Then the corresponding double EPW cube is birational to a cube of this K3 surface by \cite[Lemma 4.6]{IKKR1}.
%. 
%We conclude that \laurent{Unfinished?}
\end{proof}

%\begin{section}{72 nodal quartic sections of del Pezzo threefolds of degree 6 and Verra fourfolds}
\section{72-nodal surfaces and Verra fourfolds}\label{Verra}

Denote by $\FF l_3$ the flag manifold, obtained as a smooth hyperplane section of $\PP^2\times \PP^2$.
In this section we discuss a family of $72$ nodal surfaces contained in $\FF l_3$,  whose description is similar to the surfaces studied in this paper. %\laurent{But they come from a quadric surface bundle, so the set of nodes is not even? }
%\grzegorz{it is even because we can add O}
%At the end of the section we list some open problems. 
%Recall from \cite{IKKR} that EPW quartic sections are 
%fourfolds obtained as 
%Lagrangian degeneracy loci on a cone over $\mathbb{P}^2\times \mathbb{P}^2$. These fourfolds are singular along surfaces of degree $72$. Similarly to EPW sextics, they admit smooth double covers branched over their singular loci, which are hyperK\"ahler fourfolds called double EPW quartic sections. EPW quartic sections are naturally related to Verra fourfolds. 
%It is natural to consider one more example, the codimension two linear section of an EPW quartic \cite{IKKR}.

Recall that EPW quartics are special quartic sections of the cone over the Segre embedding of $\PP^2\times \PP^2$ in $\PP^8$, and are quotients of a \HK{} fourfold. The codimension $2$ linear section is a $72$ nodal surface contained in $\FF l_3$.
We show that this surface is a degeneracy locus of a quadric bundle on $\FF l_3$ that is birational to a complete intersection of a quadric and a determinantal cubic in $\PP^8$. 
%\fabio{We may adjust this last paragraph once we are done with section 6, so that it is clear what we have to say about the ``Verra case'' and what we don't know and it would be worth investigating.}

\begin{prop}
Let us consider the rank four bundle $T$ on $\PP^2\times \PP^2$ defined as  $$T= \Omega_{\PP^2}^1(2)\boxtimes \Omega_{\PP^2}^1(2).$$ 
%where $\pi_1$ and $\pi_2$ are projections to the factors of $\PP^2\times \PP^2$.
A hyperplane section of an EPW quartic section can be described 
as a symmetric degeneracy locus on $\PP^2\times \PP^2$ of a symmetric map from $T^{\vee}$ to $T$.
\end{prop}

\begin{proof}

Recall from \cite{IKKR} the description of an EPW quartic section $Z_{A,U_0}$ corresponding to a three-dimensional subspace $U_0$ of a six-dimensional space $W$ i.e.\ $[U]\in \Gr(3,W)$ and a subspace $A\subset \wedge^3 W$ Lagrangian with respect to the forms given by wedge product and such that $\wedge^3 U\subset  A$. Let us fix the data as above. Let $$C_{U_0}:=\{[U]\in \Gr(3,W) \colon \dim (U\cap U_0)\geq 2 \}\subset \Gr(3,W).$$ 
For $[U]\in \Gr(3,W)$ let $T_U:=\wedge^2 U\wedge W $. Then
$$Z_{A,U_0}:=\{[U]\in C_{U_0}: \dim(T_U\cap A)\geq 2 \}$$
is an EPW quartic section. Note that $\wedge^3 U_0
\subset T_U\cap A$ for every $[U]\in C_{U_0}$. 
Hence, if we define $\overline{T}_U=T_U/\wedge^3 U_0 \subset (\wedge^3 U_0)^{\perp}/\wedge^3 U_0$ and $\bar A=\bar A/\wedge^3 U_0 $, we can write $Z_{A,U_0}$ as a first Lagrangian degeneracy locus 
$$Z_{A,U_0}=\{[U]\in C_{U_0}: \dim(\overline{T}_U\cap \bar{A})\geq 1 \}.$$
Denote $P_U=T_U\cap T_{U_0}$. Let then $\widetilde{A}_U=(A\cap P_U^{\perp})/P_U$ and $\widetilde{T}_U=T_U/P_U$. We get 
$$Z_{A,U_0}\setminus [U_0] = \{[U]\in C_{U_0}\setminus [U_0] \colon \dim( \widetilde{A}_U\cap \widetilde{T}_U)\geq 1   \}.$$
Observe that $\widetilde{A}_U\cap (T_{U_0}/P_U)=
\widetilde{T}_U\cap (T_{U_0}/P_U)=0$. 
If we denote by $\widetilde {\mathcal T}$,  $\widetilde{\mathcal A}$ and $\widetilde{\mathcal T}_0$ the bundles with respective fibres $\widetilde{T}_U$, $\widetilde{A}_U$, $T_{U_0}/P_U$ over $[U]$,
then $\widetilde{ \mathcal T}\simeq   \widetilde{\mathcal A} \simeq \widetilde{ \mathcal T_0}^{\vee}$ and we deduce that $Z_{A,U_0}\setminus [U_0]$ is a symmetric degeneracy locus associated to a symmetric map $\widetilde{\mathcal T}\to \widetilde{\mathcal T}^{\vee}$. Now,  note that by construction $\widetilde{\mathcal T}$ is the quotient of the projective tangent bundle $\mathcal T_{G}$ to $\Gr(3,W)$ restricted to $C_{U_0}\setminus [U_0]$, by the projective tangent bundle $\mathcal T_{C}$ to $C_{U_0}\setminus [U_0]$. More precisely, we have the following two exact sequences:
$$0\to \mathcal O(-1) \to \mathcal T_G\to T_G(-1) \to 0,$$
$$0\to \mathcal O(-1) \to \mathcal T_C\to T_C(-1) \to 0.$$
where $T_G$ and $T_C$ are the tangent bundles to $\Gr(3,W)$ and $C_{U_0}\setminus [U_0]$ respectively.
It follows that $\widetilde{\mathcal T}=(T_G/T_C) \otimes \mathcal O(-1)$. 

Let us now restrict ourselves to a hyperplane section $C_H$ of $C_{U_0}$. It corresponds to the choice of some $U_1\subset W$ such that $U_0\cap U_1=0$; then  $$C_H=\{[U]\in \Gr(3,W)\colon \dim (U\cap U_0)=2, \ \dim (U\cap U_1)=1\}.$$ 
In particular $C_H\simeq \Gr(2,U_0) \times \Gr(1,U_1)=: G_1\times G_2$ and $$T_G |_{C_H}=\mathcal U^{\vee}\otimes \mathcal Q= (\mathcal U_{G_1}^{\vee}\oplus \mathcal U_{G_2}^{\vee})\otimes (\mathcal Q_{G_1}\oplus \mathcal Q_{G_2}),$$
where $\mathcal U_{G_1}$, $\mathcal U_{G_2}$, $\mathcal Q_{G_1}$, $\mathcal Q_{G_2}$ are the tautological and quotient vector bundles.
On the other hand $T_C$ appears in an exact sequence
$$ 0 \to T_{C_H}\to T_C|_{C_H} \to \mathcal O_{C_H}(1)\to 0,$$
and therefore $T_{C_H}=\mathcal U_{G_1}^{\vee}\otimes \mathcal Q_{G_1}\oplus \mathcal U_{G_2}^{\vee}\otimes \mathcal Q_{G_2}$
This implies that 
$$T_G/T_C=(T_G/T_{C_H})/(T_C/T_{C_H)}=\mathcal U_{G_1}^{\vee}\otimes \mathcal Q_{G_2}\oplus \mathcal Q_{G_1}\otimes \mathcal Q_{G_2}/\mathcal O_{C_H}(1)=\mathcal U_{G_1}^{\vee}\otimes \mathcal Q_{G_2}.$$
We conclude that $Z_{A,U_0}\cap H$ is  a symmetric degeneracy locus on $C_H=G_1\times G_2$ associated to a symmetric map $$\mathcal U_{G_1}^{\vee}\otimes \mathcal Q_{G_2} \otimes O_{C_H}(-1)\to \mathcal U_{G_1}\otimes \mathcal Q_{G_2}^{\vee} \otimes O_{C_H}(1).$$ Finally since both $\mathcal U_{G_1}$ and $\mathcal Q_{G_2}$ are of rank 2, the latter map is a symmetric map 
\[ T^\vee = \mathcal U_{G_1}\otimes \mathcal Q_{G_2}^{\vee} \to \mathcal U_{G_1}^{\vee}\otimes \mathcal Q_{G_2}=T.\] 
This concludes the proof.
\end{proof}

%We showed that a hyperplane section of an EPW quartic section can be described as the discriminant locus of a quadric surface bundle.
%Indeed, the map between $T^{\vee}$ and $T$ defines naturally a section of $Sym^2(T)$ thus section of $H^0(\oo_{\PP(T)}(2))$ (see \cite[\S 2]{BKK}).
We thus get a quadric bundle $\Theta$ in $\PP(T)$, that can be seen as the intersection of the image $G$ of $\PP(T)$ through $\oo_{\PP(T)}(1)$,
with a quadric. We shall see that this is a complete intersection of a quadric with a special cubic.

\begin{prop}
 The variety $\PP(T)$ admits a birational morphism onto the determinantal cubic that is the secant of $\PP^2\times \PP^2$ in its
 Segre embedding. The fibres of $\PP(T)$ map to $\PP^3$'s that are spanned by pairs of lines in the two factors of $\PP^2\times \PP^2$. The quadric bundle $\Theta\subset \PP(T)$ is thus mapped via this morphism to a quadric section of the determinantal cubic, a sixfold which is singular along a Verra threefold.
\end{prop}

    \begin{proof}
        Indeed, geometrically the bundle $\mathbb P(T)=\mathbb P(\mathcal U_{G_1}^{\vee}\otimes \mathcal Q_{G_2})$ that contains our quadric surface bundle is mapped via its $\mathcal O_{P(T)}(1)$ 
to a determinantal cubic hypersurface in $\mathbb P^8$. More precisely 
$C_H$ is a copy of $\mathbb P^2\times \mathbb P^2$  and  each fibre 
of the bundle, say over a point 
$(p,q)\in C_H$,  is mapped via $\mathcal O_{P(T)} (1)$ 
to the the span of the product $L_p\times L_q$ of the lines $p^{\perp}\subset G_1^{\vee}$ and $q^{\perp} \subset G_2^{\vee}$ in the Segre embedding of $G_1^{\vee} \times G_2^{\vee}$. 
%The symmetric map corresponding to a section of $\operatorname{Sym}^2(T)$  gives a quadric surface bundle $\Theta$ in $\mathbb{P}(T)$ which via $\mathcal O_{P(T)}(1)$  is mapped to a quadric section of the determinantal cubic.\laurent{Just a repetition of the claim!}
    \end{proof}

Consider now the restriction of $T$ to a further hyperplane section $F$ of $C_H$. Such a section corresponds to a linear isomorphism $\phi: G_1 \to G_2^{\vee}$ such that $F=\{(p,q)\in G_1\times G_2 \colon \phi(p)(q)=0\}$. After restricting the bundle $\mathbb{P}(T)$ to $\mathbb{P}_F(T)$ on $F$ we get a natural section associated to $\phi$. It is given by $(p,q)\mapsto (\phi^{T}(q),\phi(p))\in G_1^{\vee}\times G_2^{\vee}$ and is a section $s: F\to \mathbb{P}_F(T)$ as $\phi(p)(q)=0$.

This section gives rise to a conic bundle $\Gamma\to F$ obtained as the ramification locus of the fibrewise projection of $\Theta$ from $s(F)$. 
This conic bundle has however a discriminant locus with two components: one is the codimension two linear section of an EPW quartic, and  the other one is the set of  points where the section meets the quadric in the fibre. 

%The conic bundle admits the following correspondence with the Verra fourfold.

%\laurent{We use $\QQ^3$ and sometimes $Q_3$. I would personally denote 
%by $\FF l_3$ the linear section of $\PP^2\times\PP^2$, just to insist that we find nodal surfaces inside the three homogeneous threefolds $\PP^3$, $\QQ^3$, $\FF l_3$. Btw transplanting the general study of nodal surfaces from $\PP^3$ to $\FF l_3$ could also be interesting? Done} 

\section{Open problems}
\label{openProblems}

\begin{comment}
    \begin{corollary}
        The surfaces from our list above are minimal surfaces of general type and are algebraically simply connected.
    \end{corollary}
\begin{proof}
The canonical sheaf of the nodal hypersurface is ample and is computed by the adjunction formula. The canonical divisor of the covering surface is the pull back of this divisor by formulas for ramified double covers. Thus it is ample.

The case C-4 was already proved in \cite{hu21} using \cite[Cor.\ 4]{Xi87}. 
 The invariant of the surfaces are computed in the next sections thus we can show that the inequality from \cite[Cor.\ 4]{Xi87} holds in all the cases.%\laurent{Is it Corollary 4? Is it clear that our surfaces are minimal?} 
\end{proof}
\end{comment}

It is a result of \cite{ca,cataneseal} that there is a unique family of quintic surfaces with an even 
set of $16$ nodes. For sextic surfaces, Catanese and Casnati proved that even set of nodes  must have 
cardinality $24, 32, 40$ or $56$. We found $40$ nodes for the surfaces constructed in Section 4.1.
%which by \cref{codim2} are also codimension two sections of EPW sextics. Note that we have 
%$20$ parameters for the EPW sextics in $\PP^5$, hence a priori $20+8=28$ for their codimension two sections.  
%For sextics in $\PP^3$, we  have $83$ parameters, hence $68$ modulo $PGL_4$,
%which reduces to the same number $28$ when we impose $40$ ODP.  Therefore 
%we ask the following:
   
%   \begin{prob}
%       Is there a unique family of sextic surfaces with $40$ nodes? Do all such sextics come from Gushel--Mukai manifolds?
%   \end{prob}

The other nodal surfaces we found are contained in the quadric $\QQ_3$ and the flag manifold 
$\FF l_3$.

\begin{prob}
      Can one extend to these two homogeneous spaces the results that have been obtained
      for nodal surfaces in $\PP^3$ (e.g.\ \cite{beauville, cc, ct,vGZ})? 
   \end{prob}

   For example, in analogy to Catanese's result on nodal quintics in
 $\PP^3$, it is natural to ask:
 
 \begin{prob}
    Is there a unique family of even $20$-nodal quartic sections of $\QQ^3$?
\end{prob}

%In our cases the sets of nodes must necessarily be even, since we know that the 
%conic bundle structure automatically provides a nice double cover 
%(at least under some regularity conditions).

Note that it was proved by Zhao \cite{zhao} that there exists a unique family of $40$-nodal sextics in $\PP^3$. 

A closely related problem is to construct surfaces with more nodes by starting our constructions 
with nodal varieties. Indeed, the original construction by Togliatti of quintic surfaces with 
the maximal possible number of nodes, that is $31$, precisely started with cubic fourfolds 
with $16$ nodes.

\begin{prob}
Starting from nodal Gushel--Mukai fourfolds, construct  sextic surfaces in $\PP^3$ 
and quartic surfaces in $\QQ^3$ with the maximal possible number of nodes.
\end{prob} 

There exist nodal  Gushel--Mukai fourfolds with up to $20$ nodes (this is expected to be the maximal number, cf.~\cite{kv} on such a fourfold), so that one should be able to 
construct  sextics in $\PP^3$ with up to $60$ nodes and quartics in $\QQ^3$ with up to $40$ nodes (note that Schoen surfaces are $40$ nodal quartics in $\QQ^3$). Notice that Van Geemen and Zhao construct sextic surfaces with $56$ nodes \cite{vGZ}.

\medskip
 Then we would like to find more examples of even nodal surfaces of K3 type constructed from
 \HK{} manifolds with antisymplectic involutions.
In order to complete our \hyperlink{main_theorem_a}{Main Theorem A}, we would like to solve

\begin{prob}
    Is the involution on a general surface of type C-4 induced from a non-symplectic involution on 
    a LLSV eightfold? 
\end{prob}

%We infer in this way special Fano manifolds whose smoothing is related with the description of polarised families of $K3^{[2]}$ type manifolds.
In the same circle of ideas, the following problem arises in the case of Verra fourfolds.

\begin{prob}
 Describe a family of Fano manifolds that specialise to the intersection of the determinantal cubic with a quadric such that the Hodge structure of a general element is related to general \HK{} fourfolds of $K3^{[2]}$ type of Beauville--Bogomolov degree equal to $4$.
\end{prob}

Other examples potentially appear in the literature. 

\begin{prob}
    Describe the general $48$-nodal complete intersection of a cubic and a quartic that is a linear section of the fourfold described in \cite[Corollary 1.4]{CGKK} as a degeneracy locus of a quadric bundle. %\bibitem{CGKK} Chiara Camere, Alice Garbagnati, Grzegorz Kapustka, Michał Kapustka Projective models of Nikulin orbifolds arXiv:2104.09234 [math.AG]
\end{prob}

Another interesting case that was recently studied is provided by the Hessians of cubic fourfolds:
these are sextics in $\PP^5$ whose singular locus is a smooth surface of degree $35$, and they 
are naturally defined as degeneracy loci of symmetric morphisms \cite{pirola}. 
Cutting with general $\PP^3$'s in $\PP^5$
we get sextic surfaces with a "half-even" set of $35$ nodes: this is one of the cases that appear 
in \cite[Lemma 170]{cataneseal}.

\begin{comment}
 \laurent{In arxiv:2305.10871 one can find
some sextic surfaces with $35$ nodes constructed from cubic fourfolds; I wonder if they
are worth mentioning}
\giovanni{Indeed that example seems very interesting, maybe we could mention it and add a question in the final section about the cohomology of its double cover? If I get it correctly (Lemma 5.1), their surface is a smooth surface of degree 35}
\marcello{why is the set even? At a first read, I don't find in arxiv:2305.10871 any conic bundle degenerating on such surface.}
\laurent{In fact they have a quadric fourfold bundle over $\PP^5$. So by restriction to $\PP^3$ we get 
nodal surfaces, but the sets of $35$ nodes are not even, indeed. This case appears in Lemma 170 of arxiv.org/pdf/2206.05492.pdf. But the connection with cubic fourfolds is striking. Should we keep this remark?}
\end{comment}

\smallskip
Next, we observe that 
in all the cases we considered except C-4 and GM-21  we can use the Lefschetz theorem to conclude that the surfaces are topologically simply connected as sections of a   \HK{}  manifold.
The case C-4 was discussed by Ottem in the appendix of \cite{hu21}. One problem that remains is therefore:

\begin{prob}
    Is a general surface of type GM-21 or K3-35 topologically simply connected?
    Is K3-35 diffeomorphic to GM-21?
\end{prob}

Finally, it follows from \cite[Corollary 1.5]{kuznethyperbolic} that having a description of a nodal surface $S$ as a discriminant 
locus of a conic bundle on $X$ is equivalent to have a torsion sheaf on $S$ (the associated cokernel sheaf).
Thus having different descriptions of a given surface as discriminant gives different torsion sheaves.
They give different elements in the Picard group of the double cover of $S$ (the push-forward of those sheaves).

For example in \cref{CB} we described seven conic bundles on the same codimension $2$ linear section of an EPW sextic:
six corresponding to the construction GM20 and one to K3-31. We infer in this way seven associated cokernel sheaves so seven line bundles on a codimension $2$ section of the related double EPW sextic. 
This is explained by the fact that the Picard group of the codimension 2 section of the double EPW sextic have Picard rank $>1$. So there are $2$-torsion sheaves on a linear codimension 2  section of the EPW sextic that are not restrictions of a $2$-torsion sheaf on the EPW sectic. A natural problem arises.
\begin{prob}\label{pr1}
    Describe the Picard groups of the double covers of the considered nodal discriminant surfaces.
\end{prob}
   
 %  A natural question is whether the family of even $40$-nodal sextics is irreducible. There are  $24$ parameters
  % for GM-fourfolds, which give $28$ parameters if we add the choice of a hyperplane, so 
   %we could expect that the generic sextic with $40$ ODP comes from Gushel--Mukai's. 
   %But Catanese and Casnati only prove existence by exhibiting an example.

\frenchspacing
\hypersetup{urlcolor=black}
\bibliographystyle{alphaabbr}
\bibliography{nodal.bib}

 \begin{comment}
\bibliographystyle{alpha}

\end{comment}

\end{document}